\documentclass[a4paper,parskip=half]{scrartcl}
\usepackage[utf8]{inputenc}
\usepackage[english]{babel}
\usepackage{amsmath,amssymb,amsfonts,siunitx}
\usepackage{amsthm}
\usepackage{tikz,url}
\usepackage{geometry}
\usepackage{mathtools}
\mathtoolsset{showonlyrefs}
\usepackage{subcaption}
\usepackage{comment}
\usepackage{algorithm}
\usepackage{algpseudocode}
\usepackage[english,pdfusetitle,colorlinks=true,linkcolor=blue,citecolor=green!50!black,
bookmarks=true]{hyperref}
\usepackage{enumerate}
\usepackage{listings}
\usepackage{nicefrac}

\newcommand{\one}{{\mathbbm{1}}}
\newcommand{\dd}{\mathrm{d}}
\newcommand{\NN}{\mathbb{N}_0}

\theoremstyle{plain}
\newtheorem{lemma}{Lemma}[section]
\newtheorem{theorem}[lemma]{Theorem}
\newtheorem{corollary}[lemma]{Corollary}
\newtheorem{proposition}[lemma]{Proposition}
\newtheorem{remark}[lemma]{Remark}

\newtheorem{example}[lemma]{Example}
\theoremstyle{definition}

\newcommand{\C}{\mathbb{C}}
\newcommand{\R}{\mathbb{R}}
\newcommand{\N}{\mathbb{N}}

\renewcommand{\d}{\mathrm{d}}
\newcommand{\Sp}{\mathbb S}

\newcommand{\e}{\mathrm{e}}
\renewcommand{\i}{\mathrm{i}}
\newcommand{\diff}{\mathop{}\!\mathrm{d}}
\newcommand{\cM}{\mathcal{M}}
\newcommand{\cU}{\mathcal{U}}

\newcommand{\tT}{\mathrm{T}}

\usepackage{dsfont}

\DeclareMathOperator{\E}{\mathbb E}

\newcommand{\inn}[1]{\left\langle#1\right\rangle}
\normalfont

\usepackage{nicefrac} 
\usepackage{xfrac}    
\usepackage{bbm} 
\DeclareMathOperator{\loc}{loc}
\newcommand{\inv}{^{-1}}
\newcommand{\schone}{{\mathcal S(\R)}}
\newcommand{\schd}{{\mathcal S(\R^d)}}
\newcommand{\tempone}{{\mathcal S^\prime(\R)}}
\newcommand{\tempd}{{\mathcal S^\prime (\R^d)}}
\newcommand{\schonerad}{{\mathcal S_\mathrm{rad}(\R)}}
\newcommand{\schdrad}{{\mathcal S_\mathrm{rad}(\R^d)}}
\newcommand{\temponerad}{{\mathcal S_\mathrm{rad}^\prime(\R)}}
\newcommand{\tempdrad}{{\mathcal S_\mathrm{rad}^\prime (\R^d)}}
\newcommand{\Rd}{{\R^d}}

\newcommand{\tempavg}{{\mathcal A}}
\newcommand{\temprot}{{\mathcal R}}

\newcommand{\norm}{\|\cdot \|}
\newcommand{\sphere}{{\mathbb S^{d-1}}}
\newcommand{\forrall}{\quad \text{for all} \quad }
\newcommand{\cc}{{\mathcal{C}}}
\newcommand{\limn}{\lim_{n\to \infty}}

\newcommand{\lloc}{{L^1_{\loc}([0,\infty))}}

\title{Slicing of Radial Functions: \\a Dimension Walk in the Fourier Space
}

\author{Nicolaj Rux\footnotemark[1] \\%
{\footnotesize\href{mailto:rux@tu-berlin.de}{rux@tu-berlin.de}}
\and Michael Quellmalz\thanks{Technische Universität Berlin, Institute of Mathematics, Straße des 17. Juni 136, D-10623 Berlin, Germany, \url{https://tu.berlin/imageanalysis}} \\%
{\footnotesize\href{mailto:quellmalz@math.tu-berlin.de}{quellmalz@math.tu-berlin.de}}
\and Gabriele Steidl\footnotemark[1]\\%
{\footnotesize\href{mailto:steidl@math.tu-berlin.de}{steidl@math.tu-berlin.de}}}

\date{\today}

\begin{document}

\maketitle

\begin{abstract}
Computations in high-dimensional spaces can often be realized only approximately, using a certain number of projections onto lower dimensional subspaces or sampling from distributions. In this paper, we are interested in pairs of real-valued functions $(F,f)$ on $[0,\infty)$ that are related by the projection/slicing formula $F (\| x \|) = \mathbb E_{\xi} \big[ f \big(|\langle x ,\xi \rangle| \big) \big]$ for $x\in\mathbb R^d$, where the expectation value is taken over uniformly distributed directions in $\mathbb R^d$. While it is known that $F$ can be obtained from $f$ by an Abel-like integral formula, we construct conversely $f$ from given $F$ using their Fourier transforms. First, we consider the relation between $F$ and $f$ for radial functions $F(\| \cdot\| )$ that are Fourier transforms of $L^1$ functions. Besides $d$- and one-dimensional Fourier transforms, it relies on a rotation operator, an averaging operator and a multiplication operator to manage the walk from $d$ to one dimension in the Fourier space. Then, we generalize the results to tempered distributions, where we are mainly interested in radial regular tempered distributions. Based on Bochner's theorem, this includes positive definite functions $F(\| \cdot\| )$ and, by the theory of fractional derivatives, also functions $F$ whose derivative of order $\lfloor \nicefrac{d}{2}\rfloor$ is slowly increasing and continuous.

\end{abstract}

\section{Introduction}
Radial functions play an important role in approximation theory \cite{buhmann2003,W2004},
kernel density estimation \cite{P1962,R1956}, 
support vector machines \cite{ste11,SC2008}, 
kernelized principal component analysis \cite{SS2002,SC2004}, simulation of optical scattering \cite{FauKirQueSchSet23,KirQueRitSchSet21}, distance computations between probability measures \cite{GBRSS2006,szekely2002} 
as well as  dithering \cite{EhlGraNeuSte21,GraPotSte12a} in image processing, to mention only a few.
Recently, they have found applications in machine learning in connection with 
Stein variational gradient descent flows \cite{LQ2016}
and Wasserstein gradient flows \cite{AKSG2019,GBG2024,HHABCS2023}.
A central issue in the above applications is the fast evaluation of radial functions, more precisely, the computation of ``convolutions at nonequispaced knots''
\begin{equation} \label{eq:sum}
s_j\coloneqq \sum_{i=1}^N \alpha_i F(\|x_i-x_j\|), \quad j = 1,\ldots,N
\end{equation}
for large $N \in \mathbb N$,
where $x_i\in\R^d$ and $\alpha_j\in\C$.
Throughout this paper, let $\| \cdot\|$ be the Euclidean norm on $\R^d$, where the dimension becomes clear from the context.
Certain methods for high-dimensional data $x_i \in \R^d$, $d \gg 1$, were proposed in the literature.
A popular one, called random Fourier features \cite{RR2007},
was analyzed, e.g., in \cite{HasSchShiTopTraWar23,SutSch15} and found recent applications for ANOVA approximation \cite{PotWei24} and domain decomposition \cite{LiXiaLiuLia23}.

The Random Fourier features (RFF) method assumes that $F(\|\cdot\|)$ is positive definite on~$\R^d$ (see Section~\ref{sec:pos-def} for the definition), so that Bochner's theorem ensures the existence of a positive measure $\mu$ such that $F(\|\cdot\|)$ is the Fourier transform of $\mu$.
Then, 
\begin{equation*}
    F(\|x\|)= \int_\Rd \e^{-2\pi \i \langle x,\nu\rangle}\d\mu(v)\approx \frac{1}{D}\sum_{p=1}^D \e^{-2\pi \i \langle x,v_p\rangle},\quad x\in \Rd,
\end{equation*}
where $v_1,\ldots, v_D$ are iid samples of $\mu$,
and the sum in \eqref{eq:sum} can be approximated for $j=1,\ldots,N$ as
\begin{equation}\label{eq:RFF}
    s_j \approx 
\sum_{n=1}^N \alpha_n \frac{1}{D} \sum_{p=1}^D \e^{2\pi \i \langle x_j - x_n, v_p \rangle} =
\frac{1}{D} \sum_{p=1}^D \e^{2\pi \i \langle x_j, v_p \rangle} 
\sum_{n=1}^N \alpha_n \e^{-2\pi \i \langle x_n, v_p \rangle}.
\end{equation}
The inner sum over $n$ is independent of $j$ and needs $\mathcal O(DN)$ arithmetic operations. Afterwards all $s_j$ can be computed in $\mathcal O(DN)$ too, so that the total arithmetic  complexity of computing \eqref{eq:RFF}  is $\mathcal O(DN)$.
 
Another technique, called ''slicing'' \cite{hertrich2024}, is well-known in the context of optimal transport \cite{Bon23a,BRPP15,QueBeiSte23,QueBueSte24}. It reduces problem \eqref{eq:sum} from $\Rd$ to a several one-dimensional problems. It is based on the existence of a one-dimensional
function $f$ such that the slicing (projection) formula
\begin{equation} \label{eq:slice}
F (\| x \| )
=
\E_{\xi\sim \mathcal U_{\mathbb S^{d-1}}} \left[ f \left(|\langle x  ,\xi \rangle| \right) \right] \quad x\in \Rd,
\end{equation}
holds true, 
where the expectation value is taken over the uniform distribution $\cU_\sphere$ on the unit sphere $\sphere \subset \R^d$. 
In other words, a radial function $F \circ \| \cdot \|$  fulfilling  \eqref{eq:slice}
can be evaluated at $x \in \R^d$ by projecting it
onto lines of different directions $\xi$
through the origin, see Fig.\ \ref{fig:proj},
followed by evaluating a one-dimensional function  $f$
at the projected points $\langle x, \xi \rangle$.

To compute the sum $s$ with slicing, the function $F$ is approximated using $P$ quadrature points $\xi_1,\ldots, \xi_P\in{\sphere}$ via
\begin{equation} \label{eq:slicing-P}
    F(\|x\|)=\E_{\xi\sim \mathcal U_{\mathbb S^{d-1}}} \left[ f \left(|\langle x  ,\xi \rangle| \right) \right] \approx \frac{1}{P}\sum_{p=1}^P f(|\langle x,\xi_p\rangle|).
\end{equation}
Inserting this into the sum \eqref{eq:sum} yields for $j=1,\ldots,N$
the approximation
\begin{equation} \label{eq:sum-1d}
s_j
\approx
\frac{1}{P}\sum_{p=1}^P \sum_{i=1}^N \alpha_i f(|\langle x_i-x_j,\xi_p\rangle|). 
\end{equation}
For certain functions $f$, one-dimensional summations of the form \eqref{eq:sum-1d} can be
done in a very fast way, e.g.,  
via sorting or fast Fourier transforms at nonequispaced knots \cite{hertrich2024,KunPot08,PSN2004}, as done in various applications \cite{BaQue22,ChaCiuKahWei17,HofNesPip16,KirPot19}.
Extensive numerical comparisons between the slicing summation and various RFF-based approaches can be found in \cite{HerJahQue24}.
Note that \eqref{eq:sum-1d} can be viewed as an approximation of $F(\|\cdot\|)$ by the ridge functions $f(|\langle\cdot,\xi\rangle|)$, cf.\ \cite{Pin15,Uns23}, but we require it to be exact in expectation.

The relation between $F$ and its sliced version $f$ in \eqref{eq:slice}
is given by the Abel-type integral
\begin{equation} \label{eq:basis_function_intro}
F(s) 
= c_d \int_0^1 f(ts)(1-t^2)^{\frac{d-3}{2}}\,\d t
\end{equation}
with some constant $c_d$.
Note that in \cite{hertrich2024}, functions $F$ having a power series were considered to determine their slicing functions $f$.
In contrast to RFF, slicing is  not restricted to positive definite functions $F \circ \| \cdot \|$, and indeed it works also for other functions, which are of interest in applications, like
Riesz kernels ${\| \cdot\|^r}$, $r \in (0,2)$ or thin
plate splines ${\| \cdot\|^2} {\log \| \cdot \|}$.
However, we see from its integral representation \eqref{eq:basis_function_intro} 
that $F\colon[0,\infty) \to \R$ must have some smoothness properties. 
Indeed, \eqref{eq:basis_function_intro} is closely related to Riemann--Liouville fractional integrals,
and the injectivity of the transform \eqref{eq:basis_function_intro} if $f\in L^1(\R)$ as well as the inverse transform, which determines $f$ from $F$,
can be deduced via fractional derivatives, see Appendix \ref{app:sturm_liouv}. 
However, the resulting integrals are often hard to evaluate, and we will follow another approach. 
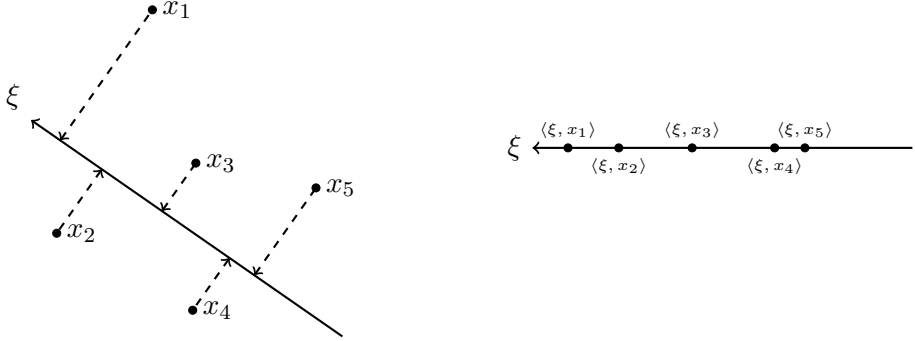
\begin{figure}
\begin{center}
\begin{tikzpicture}

\def\angle{145}
\def\length{5}
\draw[thick,->] (0,0) -- (\angle:\length) node[above left] {$\xi$};


\def\shrink{1}
\def\offsetx{7.5}
\def\offsety{2.5}
\draw[thick,->] ({\offsetx},{\offsety}) -- ({\offsetx-\shrink * \length}, {\offsety}) node[left] {$\xi$} ;

\newcommand{\ProjectPoint}[3]{

\filldraw[black] (#1:#2) circle (1.5pt) node[right] {$#3$};

\draw[dashed,<-,thick] (\angle: {cos(\angle-#1) * #2}) -- (#1:#2);

\ifnum #1 < \angle {
    \filldraw[black] ({\offsetx-cos(\angle-#1) * #2}, {\offsety}) circle (1.5pt) node[above] {\tiny $\langle \xi,#3\rangle $};
}
\else { 
    \filldraw[black] ({\offsetx-cos(\angle-#1) * #2}, {\offsety}) circle (1.5pt) node[below] {\tiny $\langle \xi,#3\rangle $};
}
\fi
}

\ProjectPoint{120}{5}{x_1}
\ProjectPoint{160}{4}{x_2}
\ProjectPoint{130}{3}{x_3}
\ProjectPoint{170}{2}{x_4}
\ProjectPoint{100}{2}{x_5}

\end{tikzpicture}
\end{center}
\caption{Projection of points $x_1,\dots,x_5\in\R^2$ onto the line in direction $\xi$.}
\label{fig:proj}
\end{figure}

In this paper, we are interested in the relation between $F$ and $f$ from a Fourier analytic point of view. More precisely, we show how $f$ can be obtained from
a radial function $F\circ\| \cdot \|$ that is the
Fourier transform of a radial $L^1$ function.
Then we will see that this is a special case
of a recovery formula for radial regular tempered distributions.
Since measures can be considered as tempered distributions, the latter one also includes positive definite functions
appearing in Bochner's relation. 
Radial tempered distributions were already considered in the literature, e.g., in \cite{Est14,GraTes12}.
However, to the best of our knowledge, our rigorous proofs 
of certain properties needed for our approach are novel.
The dimension reduction from a multivariate, radial function
$F\circ \| \cdot\|$ to a univariate one $f\circ| \cdot|$ 
in the Fourier space can be easily realized by applying a multiplication operator arising from a variable transform,
and is actually what we call ''dimension walk'', a notation borrowed from Wendland \cite[Chap. 9.2]{W2004}. 
We are completely aware that also projections onto larger than one-dimensional subspaces may be of interest, but are
out of the scope of this paper.

Outline of the paper: in Section~\ref{sec:back}
we introduce our two main players, namely the rotation operator and its inverse, the averaging operator. Then we recall the relation between the slicing formula
\eqref{eq:slice} and the Abel-like integral
\eqref{eq:basis_function_intro}.
The ''dimension walk'' is realized by a multiplication operator.
Moreover, we determine the smoothness of functions $F$ determined by the Abel-like integral.
Then, in Section \ref{sec:sliceL1}, we show as a starting point,
how the function 
$f$ in \eqref{eq:basis_function_intro} 
can be computed from a radial function $F\circ \| \cdot \|$
that is the Fourier transform of a radial $L^1$ function.
As a by-product of the smoothness result for the Abel-like integral, we will see that
the Fourier transform of a radial function in $\R^d$ is
$\lfloor \nicefrac{(d-2)}{2}\rfloor$ times continuously differentiable,
a result that should be known in the literature, although we
did not find a direct reference.
In Section~\ref{sec:walk_c}, we first recall the definition of radial Schwartz functions and prove that the averaging and  rotation operator are continuous operators on these spaces. 
This allows to generalize the reconstruction to radial regular tempered distributions.
Clearly, this is more general than the approach in the previous section 
and we provide in particular two examples.
Since measures can be treated as special tempered distributions, 
we obtain an inversion formula that is valid for all positive definite radial functions $F\circ \| \cdot \|$
based on Bochner's theorem.
Conclusions are drawn in Section~\ref{sec:conclusions}.
Auxiliary technical results are postponed to the appendix.

\section{Rotating, Averaging and Slicing} \label{sec:back}
We denote by $\mathcal C(\R^d)$ the space of complex-valued continuous functions, by
$\mathcal C_b(\R^d)$ the Banach space of bounded continuous functions, 
and by 
$\mathcal C_0(\R^d)$ 
the Banach space of continuous functions vanishing for $\|x\| \to \infty$ with the norm
\begin{equation} \label{eq:infty-norm}
\|\Phi\|_\infty \coloneqq \sup_{x \in \R^d} |\Phi (x)|.
\end{equation}
Let $\mathcal C^\infty(\R^d)$ be the space of infinitely differentiable functions.
Further, let 
$L^p_\mathrm{loc}(\R^d)$, $p\in[1,\infty)$, 
denote the space of locally $p$-integrable functions and $L^\infty_\mathrm{loc}(\R^d)$ the space of locally bounded functions.

We are interested in radial functions $\Phi\colon\R^d \to \R$, which are characterized by the property that for all $x \in \R^d$,
\begin{equation}
 \Phi(x) = \Phi(Q x)   \quad \text{for all}  \quad  Q \in \text{O}(d) ,
\end{equation}
where  $\mathrm{O}(d)$ denotes the set of orthogonal $d\times d$ matrices.
We need two operators. The \emph{rotation operator} $\mathcal R_d$ associates to  $F\colon [0,\infty) \to \R$ the radial function $\mathcal R_d F\colon \R^d \to \R$ given by
\begin{equation} \label{op:rot}
\mathcal R_d F\coloneqq F\circ \|\cdot \|.
\end{equation}
Since every function $F\colon[0,\infty)\to\R$ can be identified with its \emph{even} continuation $F:\R\to\R$, we define $\mathcal R
_d$ alternatively for all even functions on $\R$.
Every radial function is of the form \eqref{op:rot}. 
The \emph{spherical averaging operator} $\mathcal A_d$ assigns to a function $\Phi\colon \R^d \to \R$, which is integrable on every sphere $r\sphere$, $r>0$, the
function
$\mathcal A_d\Phi\colon \R \to \R$ defined by
\begin{equation} \label{op:ave}
\mathcal A_d \Phi (r)\coloneqq  
\frac{1}{\omega_{d-1}} \int_\sphere \Phi(r \, \xi) \, \d\xi = \E_{\xi\sim \mathcal U_{\mathbb S^{d-1}}} \left[ \Phi(r \, \xi) \right]\forrall r\in\R.
\end{equation}
Note that as soon as $\Phi$ is continuous on $\Rd\setminus \{0\}$ or $\Phi$ is radial, the function $\mathcal A_d \Phi$ is well-defined. By definition $\mathcal A_d \Phi$ is an even function, i.e.,
$\mathcal A_d \Phi(r) = \mathcal A_d \Phi(-r)$, $r \in \R$
and we have
for $d=1$ that 
$$ \mathcal A_1\Phi(r)=\tfrac12 \left(\Phi(r)+\Phi(-r) \right) \forrall r \in \R.$$
Moreover, we obtain by definition that
$$
\mathcal A_d (\Phi \circ Q) = \mathcal A_d \Phi \quad \text{for all} \quad Q \in 
\text{O} (d). 
$$
The operator $\mathcal A_d$ is the inverse of $\mathcal R_d$, meaning that for every even function $F\colon\R\to\R$ and $r\geqslant 0$ it holds
 \begin{equation} \label{eq:inv_ar}
        (\mathcal A_d \circ \mathcal R_d) F(r)
        =\E_{\xi\sim\cU_\sphere} \left[F(\|r\xi\|) \right]
        =F(|r|)
        =F(r),
 \end{equation}
 and conversely for every radial function $\Phi\colon\R^d\to\R$ and $x\in \Rd$ we have
 \begin{equation} \label{eq:inv_ra}
 (\mathcal R_d \circ \mathcal A_d) \Phi(x)
 =  (\mathcal A_d \Phi)(\|x\|) 
        =\E_{\xi\sim\cU_\sphere} \left[\Phi(\|x\| \xi) \right]
        =\Phi(x).
  \end{equation}
The following theorem considers special radial functions of the form \eqref{eq:slice}.

\begin{theorem} \label{thm:slicing_b}
Let $d \in \mathbb N$, $d \ge 2$ and let $f \in L^1_\mathrm{loc}([0,\infty))$. 
Then the function $F\colon [0,\infty) \to \mathbb R$ 
fulfilling the slicing relation
\begin{equation} \label{slicing_def}
\mathcal R_d F
=
\E_{\xi\sim \mathcal U_{\mathbb S^{d-1}}} \left[ f \left(|\langle \cdot ,\xi \rangle| \right) \right] 
\end{equation} 
is determined by the Abel-type integral
\begin{equation} \label{eq:basis_function}
F(s) 
= c_d \int_0^1 f(ts)(1-t^2)^{\frac{d-3}{2}}\,\d t
=
c_d \, \frac{1}{s} \, \int_0^s f(t)\Big(1-\frac{t^2}{s^2}\Big)^{\frac{d-3}{2}}\,\d t,
\end{equation}
where $c_d \coloneqq \frac{2 \omega_{d-2}}{\omega_{d-1}}$ and $\omega_{d-1}= \frac{2\pi^{\nicefrac{d}{2}}}{\Gamma(\nicefrac{d}{2})}$ denotes the surface measure of $\mathbb S^{d-1}$.
\end{theorem}

In the context of slicing, the last theorem was recently proved in \cite{hertrich2024}. However, as outlined in the next Remark~\ref{rem:Radon}, the slicing relation \eqref{slicing_def} equals the adjoint Radon transform applied to a radial function. In this context, Theorem~\ref{thm:slicing_b} constitutes a special case of \cite[Lem.~2.1]{Rub04}.
For convenience, we add the short proof of Theorem \ref{thm:slicing_b} in Appendix \ref{app:1}.

\begin{remark}[Connection with the Radon transform] \label{rem:Radon}
    The Radon transform assigns to a function $\Phi\colon\R^d\to\R$ its integrals along all hyperplanes and is defined by
    $$
    R\Phi(\xi,t)
    \coloneqq
    \int_{\xi^\perp} \Phi(x+t\xi) \d x, \qquad \xi\in\mathbb S^{d-1},\ t\in\R,
    $$
    where $\d x$ is the integration along the hyperplane $\xi^\perp\coloneqq\{x\in\R^d: \langle x,\xi\rangle=0\}$.
    The dual or adjoint Radon transform $R^*$ is given for $\Psi\colon\mathbb S^{d-1}\times\R\to\R$ by
    $$
    R^*\Psi(x)
    \coloneqq
    \int_{\mathbb S^{d-1}} \Psi(\xi,\langle x,\xi\rangle) \d\xi, \qquad x\in\R^d,
    $$
    see \cite{NatWue01}. If $\Psi$ is radial, i.e., independent of its first argument, we can write it as $\Psi(\xi,t)=f(t)$ for some function $f\colon\R\to\R$. Hence, the dual Radon transform of such radial function is exaclty the slicing relation \eqref{eq:slice} up to the constant $\omega_{d-1}$.
    Note that also formulas for the Radon transform of radial functions are well-known and of similar the structure as \eqref{eq:basis_function}, but not identical, see \cite[proof of Thm 2.6]{Hel10}.
    An inversion formula for the adjoint Radon transform based on the Riesz potential was shown in \cite{Sol87} and \cite[Thm.~8]{Uns23}.
\end{remark}

The following theorem, whose proof is given in Appendix \ref{app:walk_a}, clarifies smoothness properties of the function $F$ in the Abel-type integral.

\begin{theorem} \label{thm:slicing}
For $d \in \mathbb N$ with $d \ge 3$, 
let $f \in L^1_\mathrm{loc}([0,\infty))$ for odd $d$ and $f \in L^p_\mathrm{loc}([0,\infty))$ with $p>2$ for even $d$.
Then the function $F$ defined by \eqref{eq:basis_function}
is $\lfloor \nicefrac{(d-2)}{2} \rfloor$-times continuously differentiable on $(0,\infty)$.
Moreover, if $d$ is odd, then the $\lfloor \nicefrac{(d-2)}{2}  \rfloor$-th derivative of $F$ is absolutely continuous. This result is optimal in the sense that there exists a function $f\in L^\infty_\textup{loc}([0,\infty))$ such that $F$ is not $\lfloor \nicefrac{d}{2}\rfloor$ times differentiable.
\end{theorem}

In \eqref{eq:basis_function}, two integral formulas for \( F \) are provided. The first formula can be used to show that \( F \in \mathcal{C}^n((0, \infty)) \) if \( f \in \mathcal{C}^n((0, \infty)) \), by applying the Leibniz rule for differentiation under the integral sign.
The second formula is useful when \( f \) is not differentiable. In this case, the smoothness of \( F \) is determined by the term \( (1 - \nicefrac{t^2}{s^2})^{\frac{d-3}{2}} \). The smoothness of \( F \) is then determined by the number of integrable derivatives of \( (1 - \nicefrac{t^2}{s^2})^{\frac{d-3}{2}} \). Since the exponent \( \nicefrac{(d-3)}{2} \) causes \( (1 - \nicefrac{t^2}{s^2})^{\frac{d-3}{2}} \) to exhibit a different number of integrable derivatives depending on whether \( d \) is even or odd, we make this distinction in Theorem~\ref{thm:slicing}.
The case \( d = 2 \) is special, because here \( \nicefrac{(d-3)}{2} = -\nicefrac{1}{2} \) is negative and \( (1 - \nicefrac{t^2}{s^2})^{-1/2} \) has a pole at \( s \). 

Using the theory of fractional integrals and derivatives it is almost possible to reverse Theorem \ref{thm:slicing}. However we need to assume $\lfloor \nicefrac{d}{2}\rfloor$ instead of $\lfloor\nicefrac{(d-2)}{2}\rfloor$ continuous derivatives of $F$ in order to guarantee that $F$ can be sliced. The statement and proof of this claim can be found in Theorem \ref{thm:enusre_slicing}
in the appendix.

In the rest of this paper, we are interested in characterizing $f$ from given $F$, i.e., the inversion of the Abel-like integral transform \eqref{eq:basis_function},
where we want to use Fourier analytic tools.

\section[Slicing of integrable functions]{Slicing of $L^1$ Functions} \label{sec:sliceL1}
We start with functions $F$ that are
Fourier transforms of absolutely integrable functions.
Let $L^p(\R^d)$, $p \in [1,\infty)$, be the Banach space of
$p$-integrable functions. 
The \emph{Fourier transform} $\mathcal F_d\colon L^1(\mathbb R^d) \to \mathcal C_0(\mathbb R^d)$ is an injective, linear operator
defined for $\Phi \in L^1(\mathbb R^d)$ by 
\begin{equation} \label{fourier_1}
\hat \Phi = \mathcal F_d [\Phi] \coloneqq \int_{\R^d} \e^{-2\pi \i \langle x,\cdot \rangle} \Phi(x) \, \d x.
\end{equation}
If $\hat \Phi \in L^1(\mathbb R^d)$, then the inverse Fourier transform reads as
\begin{equation} \label{fourier_2}
\Phi = \mathcal F_d^{-1} [\hat \Phi] \coloneqq \int_{\R^d} \e^{2\pi \i \langle \cdot,v \rangle} \hat \Phi(v) \, \d v.
\end{equation}
On even functions, and in particular radial functions, the Fourier transform coincides with its inverse.
Moreover, the Fourier transform of a real-valued, radial function is real-valued again, and we have 
$\widehat {\Phi \circ Q} =  \hat \Phi \circ Q$ for all  $Q \in \text{O}(d).
$
For $\rho\colon \R\to \R$, we define
the \emph{multiplication operator} by
\begin{equation} \label{op:mult_m}
\mathcal M_d \rho(r) \coloneqq \rho(|r|) |r|^{d-1} 
\forrall r\in \R.
\end{equation}
By definition, $\mathcal M_d \rho$ is an even function.
We obtain the  following novel inversion result.

\begin{proposition} \label{thm:walk_1}
Let $d \ge 3$. Assume that  $F\colon [0,\infty) \to \R$ fulfills 
$\mathcal{R}_d F = \mathcal F_d [\mathcal R_d \rho]$ 
for some function
$\rho\colon[0,\infty) \to \R$ 
with 
$\mathcal{R}_d \rho \in L^1 (\R^d)$.
Then the function $f\colon[0,\infty) \to \R$  given by the corresponding even function
\begin{equation} \label{eq:walk_a}
f = \tfrac{\omega_{d-1}}{2} ( \mathcal F_1 \circ\mathcal M_d)[\rho] \in \mathcal C_{0}(\R),
\end{equation}
fulfills \eqref{slicing_def},
where $\rho$ is also considered evenly extended here.
If in addition $\mathcal{R}_d F \in L^1(\R^d)$, then
\begin{equation} \label{eq:walk_aa}
f = \tfrac{\omega_{d-1}}{2} (\mathcal F_1\circ \mathcal M_d \circ \mathcal A_d \circ \mathcal F_d^{-1} \circ \mathcal R_d ) [F] .
\end{equation}
\end{proposition}

\begin{proof}
Using that $v = \|v\| \xi$, where $\xi \in \sphere$, we obtain by assumption 
\begin{align}
\mathcal R_d F (x)
&=   \int_{\R^d} \e^{-2\pi \i \langle x,v \rangle} \, \rho (\|v\|) \,\d v\\
&= \int_{\Sp^{d-1}} \int_0^\infty \e^{-2\pi \i \langle x,\xi \rangle r} \, \rho(r) r^{d-1} \, \d r \,\d \xi\\&
= \frac12 \int_\sphere \int_\R \e^{-2\pi \i \langle x,\xi\rangle r} \rho(r)|r|^{d-1}\,\d r \,\d \xi\\&
=\frac12 \int _ \sphere \mathcal F_1[ \mathcal M_d\rho]({|\langle x,\xi\rangle|} )\, \d\xi\\
&= 
\E_{\xi\sim \mathcal U_{\mathbb S^{d-1}}} \left[ \tfrac{\omega_{d-1}}{2} ( \mathcal F_1 \circ\mathcal M_d)[\rho](|\langle x,\xi\rangle |)
 \right].
\end{align}
On the other hand, we have by Theorem \ref{thm:slicing_b} that $f$ with \eqref{eq:basis_function}
fulfills
\begin{equation}
\mathcal R_d F (x)
=
\E_{\xi\sim \mathcal U_{\mathbb S^{d-1}}} \left[ f \left(|\langle x  ,\xi \rangle |\right) \right].
\end{equation}
This implies that $f=\tfrac{\omega_{d-1}}{2} ( \mathcal F_1 \circ \mathcal M_d)[\rho]$ fulfills 
 \eqref{slicing_def}.
 Since $\mathcal R_d\rho\in L^1(\Rd)$, we know that $\mathcal M_d\rho\in L^1(\R)$ and hence its Fourier transform is continuous, so that $f\in \mathcal C _ 0(\R)$ is even.

If in addition $\mathcal{R}_d F \in L^1(\R^d)$, then 
$\mathcal R_d \rho = \mathcal F_d^{-1}[\mathcal R_d F]$ 
and by \eqref{eq:inv_ar} further
$\rho = (\mathcal A_d \circ \mathcal F_d^{-1} \circ \mathcal R_d) F$. Plugging this into \eqref{eq:walk_a},
we obtain the second assertion.
\end{proof}

\begin{remark}\label{rem:otherinvthm}
The last theorem can be seen as the analogue of the Fourier slice theorem \cite{NatWue01}, which connects the Fourier transform of a $d$-dimensional function with the one-dimensional Fourier transform of its Radon transform. However, our theorem is a slice theorem for the adjoint Radon transform of radial functions, cf.\ Remark~\ref{rem:Radon}.
An alternative proof of \eqref{eq:walk_a} can be derived from Solomon's inversion formula of the adjoint Radon transform \cite[Eq.~(8.2)]{Sol87}, which is stated in terms of the Calderón--Zygmund operator \cite[Def.~1]{CalderonZygmund}. 
Using the Fourier slice theorem combined with the equality \cite[E1.~(2.18)]{Sol87}, equation \eqref{eq:walk_a} can be recovered for certain Schwartz functions. 
\end{remark}

Let us note that another characterization of Fourier transforms of radial 
$L^1$ functions is given by the following remark, see, e.g.\ \cite{grafakos2009}.

\begin{remark}
The Fourier transform of a radial function $\mathcal{R}_d \rho \in L^1 (\R^d)$, $d \ge 2$ is also a radial function and can be written as 
$$
\mathcal F_d \big[\mathcal{R}_d \rho \big] (x)
= 2\pi \|x\|^{1-\nicefrac{d}{2}} \int_0^\infty \rho(r) r^{\nicefrac{d}{2}} J_{\nicefrac{d}{2}-1}(2 \pi r \|x\|) \, \d r, 
$$
where $J_{\nicefrac{d}{2}-1}$ denotes the Bessel function of first kind of order $\nicefrac{d}{2} -1$.
\end{remark}

Combining Theorem \ref{thm:slicing} and Proposition \ref{thm:walk_1} gives the following corollary.

\begin{corollary}\label{cor:smooth}
The Fourier transform of any radial function from $L^1 (\R^d)$
is $\lfloor \nicefrac{(d-2)}{2} \rfloor$ times continuously differentiable on $\R^d\setminus\{0\}$. 
\end{corollary}

\begin{proof}
    Let $\Phi = \mathcal R_d \rho \in L^1(\Rd)$ for some $\rho\colon[0,\infty) \to \R$. 
    Then the Fourier transform $\mathcal F_d[\Phi]$ exists and is radial. Therefore, a function $F\colon[0,\infty)\to \R$ exists with $\mathcal R_d F= \mathcal F_d[\mathcal R_d \rho]=\mathcal F_d[\Phi]$. By Proposition \ref{thm:walk_1}, it follows that $f$ and $F$ satisfy \eqref{eq:basis_function} as well as $f\in \mathcal C(\R)$. By Theorem \ref{thm:slicing} the function $F$ is $\lfloor \nicefrac{(d-2)}{2} \rfloor$ times continuously differentiable on $(0,\infty)$. The smoothness of the Euclidean norm on $\Rd\setminus\{0\}$ yields the assertion.
\end{proof}

The ``dimension walk'' between Fourier transforms of radial functions in different dimensions
was discussed, e.g. in \cite{Est14,GraTes12}.
Various examples of sliced transform
pairs $(F,f)$ were given in  \cite{hertrich2024,HerJahQue24,Rub04}.
Here are two interesting  ones.

\begin{example} \label{example}
i) If $F(x)= \exp\big(-\frac{x^2}{2}\big) $ is a Gaussian, then the Fourier transform $\rho(x) = (2\pi)^{d/2} \exp(-2\pi^2x^2) $ in Proposition \ref{thm:walk_1} is a Gaussian as well
and we obtain that
$f = \frac{\omega_{d-1}}{2} \mathcal F_1 [\mathcal M_d\rho] = \frac{\omega_{d-1}}{2} {\mathcal F_1[\rho(|\cdot|) |\cdot|^{d-1}]}$ on $\R$ or conversely 
$\mathcal F_1 f = \mathcal F_1^{-1} f = \mathcal M_d \rho$.
We have
$$F(x) = \exp\bigg(-\frac{x^2}{2}\bigg) \quad \Longleftrightarrow \quad f(x) = {}_1F_1 \bigg(\frac{d}{2},\frac{1}{2},-\frac{x^2}{2} \bigg). $$
where the confluent hypergeometric distribution function ${}_1F_1$ is defined by \cite[(15)]{BARNARD1998128}
\begin{align*}
{}_1F_1(a,c,z)\coloneqq \sum_{n=0}^\infty \frac{(a)_n}{(c)_nn!}z^n, \text{ where } (a)_n\coloneqq \frac{\Gamma(a+n)}{\Gamma(a)}.
\end{align*}
The graphs of these functions 
are depicted for dimension $d=10$ in Fig.\ \ref{fig:graphs}.
The function~$\mathcal{R}_d F$ is positive definite in every dimension $d \in \N$, and the function 
$f$ is  positive definite in one dimension by  Bochner's theorem \ref{thm:bochner}.
However, $f$ shows oscillations, which increase with the dimension.
Its Fourier transform $\mathcal F_1 f$ has only two modes, which get separated more far from each other with an increasing dimension, but 
keep their shapes.
 
\begin{figure}
\begin{center}
\begin{tabular}{ccc}
\includegraphics[width=0.25\textwidth]{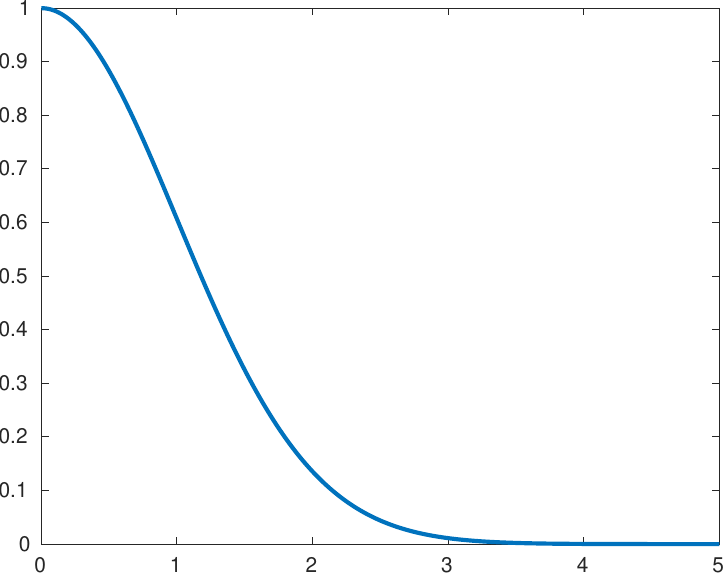} \hspace{0.5cm}&
\includegraphics[width=0.25\textwidth]{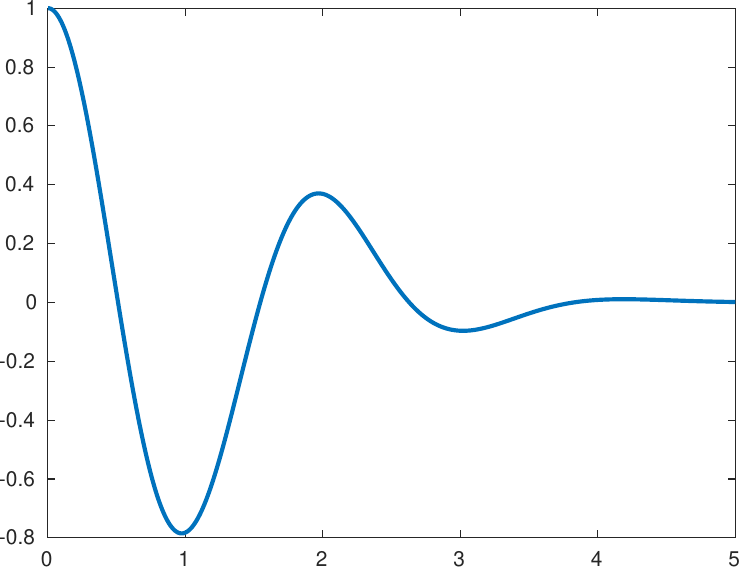} \hspace{0.5cm}&
\includegraphics[width=0.25\textwidth]{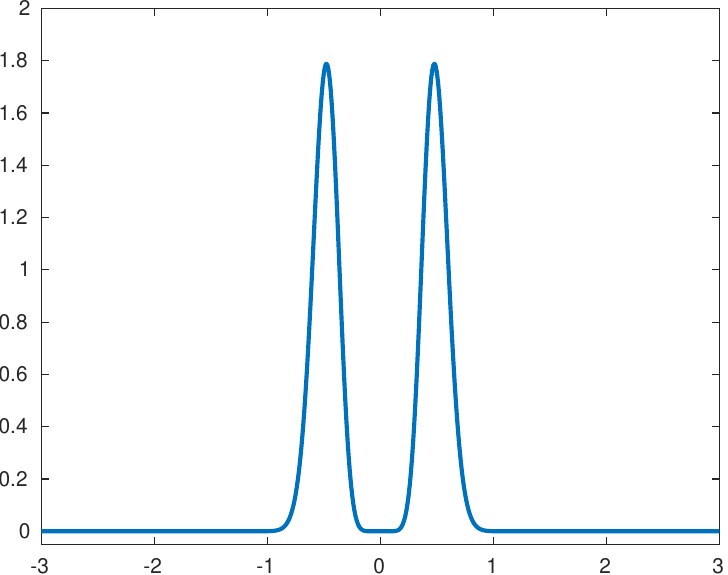}\\
$F$\hspace{0.5cm}&$f$\hspace{0.5cm}&$\mathcal F_1 f$
\end{tabular}
\end{center}
\caption{ Abel-type transform of the confluent hypergeometric distribution $f$ results in the  Gaussian $F$ ($d=10$).}
\label{fig:graphs}
\end{figure}

ii) For Riesz kernels, both $F$ and $f$ have the same structure, 
more precisely, for $r >-1$,
we have
\begin{equation} \label{eq:Riesz}
F(x) =  x^r  \quad \Longleftrightarrow \quad f(x) = \frac{\sqrt\pi \Gamma(\frac{d+r}{2})}{\Gamma(\frac{d}{2})\,\Gamma(\frac{r+1}{2})} \, x^r.
\end{equation}
\end{example}

\section{Slicing of Tempered Distributions} \label{sec:walk_c}

In this section, we extend our considerations to radial tempered distributions. 
In the first two subsections, we present the theory about radial Schwartz functions and respective distributions,
where we partially build up on results in \cite{GraTes12}.  
For a general background on (tempered) distributions, we refer to \cite{GelShi64,PlPoStTa23}.
This allows to show in Subsection~\ref{sec:distributional_slicing} a generalization of the slicing Proposition~\ref{thm:walk_1}.

\subsection{Radial Schwartz Functions}

For a smooth function $\varphi\in \cc^\infty(\Rd)$ and an integer $m\in \N$, define 
\begin{equation}\label{eq:S}
    \|\varphi\|_m\coloneqq \max_{\beta\in \N^d, |\beta|\leqslant m} \|(1+\|\cdot\|)^m D^\beta \varphi\|_\infty,
\end{equation}
where $ |\beta|\coloneqq \sum_{k=1}^d \beta_k$.
The space of \emph{Schwartz functions} $\schd$ consists of all smooth functions $\varphi\in \cc^\infty(\Rd)$ such that $\|\varphi\|_m<\infty$ for all $m\in \N$. 
A sequence $\varphi_n\in \schd$ converges to $\varphi\in \schd$ if
\begin{equation*}
    \limn \|\varphi_n-\varphi\|_m = 0 \forrall m\in \N.
\end{equation*}
The Fourier transform \eqref{fourier_1}
is a linear, bijective and continuous map from $\schd$ to $\schd$.
A Schwartz function $\varphi\in \schd$ is called \textit{radial} if $\varphi=\varphi\circ Q$ for all $Q\in \mathrm{O}(d)$.
The \emph{space of radial Schwartz functions} is denoted by 
$$\schdrad \coloneqq \{ \varphi\in \schd:  
\varphi\circ Q=\varphi \text { for all } Q\in \mathrm{O}(d)  \}.$$
In particular, for $d=1$, we obtain the even Schwartz functions $\schonerad$.

The following theorem shows that the rotation and the averaging operator are well-defined on radial Schwartz functions, i.e., they map radial Schwartz functions to radial Schwartz functions on $\R$ and $\R^d$. We will use $\psi$  to denote one-dimensional Schwartz functions and
$\varphi$ to address $d$-dimensional Schwartz functions.

\begin{theorem}\label{prop def of avg} 
\begin{enumerate}[i)]
\item The rotation operator $\mathcal R_d\colon \schonerad \to \schdrad$ given  by \eqref{op:rot}
    is linear and continuous. In particular, there exist constants $b_m>0$ 
    such that $\|\mathcal R_d \psi\|_m \leqslant b_m\|\psi\|_{4m}$ for all $\psi\in \schonerad$ and all $m \in \mathbb N$.
\item    The averaging operator $\mathcal A_d \colon \schd\to \schonerad$ given by \eqref{op:ave}
    is well-defined and continuous with 
    $\|\mathcal A_d \varphi \|_m\leqslant d^m \|\varphi\|_m$ for all $\varphi\in \schd$ and all $m\in \N$. 
\end{enumerate}    
\end{theorem}

A rough sketch of the proof can be found in \cite{GraTes12}. We give a rigorous proof in Appendix~\ref{app:avg}.
The proof of i) is technically more involved than ii). 
We start with a one-dimesnional function $\psi\in \schonerad$ and estimate the derivatives of the $d$-dimensional function $\mathcal R_d\psi(x)=\psi(\|x\|)$. At the origin, we cannot apply the chain rule to deduce smoothness of $\mathcal R_d\psi$, as the norm is not differentiable there. 
Note that we estimate $\|\mathcal R_d\psi\|_m$ by a multiple of $\|\psi\|_{4m}$ rather then $\|\psi\|_m$ due to the curvature of the norm.
Consider for example $\psi$ to be linear on some interval $I$, then $\psi''$ vanishes on $I$, but the second order derivatives of $\mathcal R_d\psi$ do not.

Part ii) is proven by estimating the derivatives of $ r\mapsto \varphi(r \xi)$. If $\xi$ is a unit vector $e_j$, then $\frac{\d}{\d r}\varphi(re_j)=\frac{\partial}{\partial e_j}\varphi(re_j)$ and we estimate it by $\|\varphi\|_1$. Otherwise, we find a rotation matrix $Q$ that maps $\xi$ to $e_1$. Then the factor $d^m$ comes in as estimate of the derivatives of the rotation $Q$.

The following corollary is a direct consequence of the above theorem and the relations
\eqref{eq:inv_ar} and \eqref{eq:inv_ra}.
It shows that $\mathcal A_d$ restricted to $\schdrad$ is bijective.

\begin{corollary}\label{corollary avg inv is rot d}
\begin{enumerate}[i)]
\item 
    The concatenated operator 
    $
   \mathcal R_d \circ \mathcal A_d \colon \schd\to \schdrad
    $
    is a continuous projection onto $\schdrad$, i.e., it is surjective and $(\mathcal R_d \circ \mathcal A_d) ^2=\mathcal R_d \circ \mathcal A_d$.
\item     The operator $\mathcal A_d \colon \schdrad\to \schonerad$ is a homeomorphism, i.e., it is bijective and continuous with continuous inverse, 
$$
\mathcal A_d \circ \mathcal R_d
= \mathrm{Id}_{\schonerad}
\quad \text{and} \quad
\mathcal R_d \circ \mathcal A_d
= \mathrm{Id}_{\schdrad}.
$$

\end{enumerate}
\end{corollary}

\subsection{Radial Tempered Distributions}
The space of \emph{tempered distributions} $\tempd$ consists of all continuous linear functionals on $\schd$. A sequence $T_k\in \tempd$ converges to $T\in \tempd$ if $\lim_{k\to \infty} \langle T_k,\varphi \rangle = \langle T,\varphi \rangle$ for all $\varphi \in \schd$.
The Fourier transform $\mathcal F_d\colon \tempd \to \tempd$
is the linear, continuous operator defined for a tempered distribution $T\in\tempd$   by 
\begin{equation} \label{eq:Fourier temp}
\inn{\mathcal F_d T,\varphi} \coloneqq \inn{T,\mathcal F_d\varphi} 
\quad \text{for all} \quad \varphi\in\schd.
\end{equation}
A distribution $T\in \tempd$ is called \textit{radial} if 
$$
\langle T,\varphi\circ Q\rangle=\langle T,\varphi\rangle \quad \text{for all} \quad Q\in \mathrm{O}(d),\ \varphi\in\schd.
$$ 
The \emph{space of all radial tempered distributions} is denoted by 
$$\tempdrad\coloneqq\{ T\in\tempd: 
\langle T,\varphi\circ Q\rangle =\langle T,\varphi\rangle
\text{ for all } Q\in \mathrm{O}(d) \text{ and all } \varphi\in \schd \}.
$$
Since we obtain for $T\in\tempd$ and every $\varphi \in \schd$
that
$$
\langle \mathcal F_d T, \varphi\circ Q\rangle
=
\langle  T, \mathcal F_d(\varphi\circ Q) \rangle
=
\langle  T, (\mathcal F_d \varphi) \circ Q \rangle
=
\langle  T, \mathcal F_d \varphi  \rangle
=
\langle \mathcal F_d T, \varphi\rangle,
$$
we conclude that the Fourier transform of a radial tempered distribution is again radial.

The proof of the following lemma can be found in \cite[Prop.~3.2]{GraTes12}.
Based on Corollary \ref{corollary avg inv is rot d}\,i), the operator
$\mathcal R_d \circ \mathcal A_d$ maps  $\schd$ onto $\schdrad$,
so that the lemma finally states that radial distributions can be determined by testing just with radial Schwartz functions.

\begin{lemma}\label{lemma invariand under Rad}
    For $T\in\tempdrad$ and every $\varphi \in \schd$, it holds
       $\langle T,(\mathcal R_d \circ  \mathcal A_d) \varphi \rangle =\langle T, \varphi\rangle $.
      In particular, a radial distribution is uniquely determined by its application to radial Schwartz functions.
\end{lemma}

Next, we consider analogues of the operators $\mathcal A_d$ and $\mathcal R_d$ for distributions. As with the Fourier transform, they are constructed via the application of an ``adjoint'' operator to Schwartz functions. We define the  operator
$\temprot_d^\star\colon \tempd\to \temponerad$ by
\begin{equation} \label{eq:Rd-star}
\langle \temprot_d^\star T , \psi \rangle \coloneqq 
\langle  T , (\mathcal R_d \circ \mathcal A_1)\psi \rangle
\quad \text{for all} \quad \psi \in \mathcal S(\R).
\end{equation}
Indeed, $\temprot_d^\star T \in \temponerad$ for every $T \in \tempd$, since we have for every  $\psi \in \mathcal S(\R)$ that
$$
\langle \temprot_d^\star T , \psi(- \cdot) \rangle 
= 
\langle  T , \mathcal R_d \left( \mathcal A_1 \psi(-\cdot) \right) \rangle
=
\langle  T , \mathcal R_d \left( \mathcal A_1 \psi \right) \rangle
= \langle \temprot_d^\star T , \psi \rangle.
$$
Further, we introduce the operator
$\mathcal A_d^*\colon \tempone\to\tempdrad$ by
\begin{equation} \label{eq:Ad-star}
\langle \tempavg_d^* \tau , \varphi \rangle 
\coloneqq \langle  \tau , \mathcal A_d \varphi \rangle
\quad \text{for all} \quad \varphi \in \schd,
\end{equation}
Clearly, $\tempavg_d^* \tau \in \tempdrad$ for every $\tau \in \tempone $, since
we get for every $Q \in \mathrm{O}(d)$ and every $\varphi \in \schd$ that
$$
\langle \tempavg_d^* \tau , \varphi \circ Q \rangle
=
\langle \tau , \mathcal A_d(\varphi \circ Q) \rangle
=
\langle \tau , \mathcal A_d \varphi  \rangle
=
\langle \tempavg_d^* \tau , \varphi \rangle.
$$
Note that $\mathcal A_d^*$ is the adjoint operator of $\mathcal A_d\colon\schd\to\schonerad$, while $\mathcal R_d^\star$ is the radial extension of the adjoint of $\mathcal R_d\colon\schonerad\to\schdrad$ to $\schone$.
In the following, we will use  $\tau$ to denote one-dimensional  distributions and
$T$ for $d$-dimensional  distributions.
The next proposition shows that $\temprot_d^\star$ and $\tempavg_d^*$ become bijective when restricted to radial distributions.

\begin{lemma}\label{Lemma Rot inverse to avg}
The restrictions 
$\temprot_d^\star \colon \tempdrad\to\temponerad$
and
$\tempavg_d^*\colon \temponerad\to \tempdrad$
are bijective and inverse to each other, i.e.,
$$
\temprot_d^\star \circ \tempavg_d^* = \mathrm{Id}_\tempdrad
\quad \text{and} \quad 
\tempavg_d^* \circ \temprot_d^\star = \mathrm{Id}_\temponerad.
$$
\end{lemma}

\begin{proof}
Let $T\in \tempdrad$.
For all $\varphi\in \schd$, we see that $\tempavg_d \varphi$ is even and obtain by Lemma~\ref{lemma invariand under Rad} that
\begin{equation*}
    \langle (\tempavg_d^* \circ \temprot_d^\star) T, \varphi\rangle
     =
     \langle T, (\mathcal R_d \circ \mathcal A_1\circ  \tempavg_d) \varphi\rangle
     =
     \langle T, (\mathcal R_d \circ  \tempavg_d) \varphi\rangle
       =\langle T, \varphi\rangle.
\end{equation*}
Let $\tau\in \temponerad$.
Then, it follows
for all $\psi \in \mathcal S(\R)$ 
by Corollary \ref{corollary avg inv is rot d}\,ii) that
\begin{equation*}
    \langle (\temprot_d^\star \circ \tempavg_d^*) \tau, \psi\rangle
    =
    \langle \tau, (\tempavg_d \circ \temprot_d \circ \mathcal A_1) \psi\rangle 
    =
    \langle \tau, (\tempavg_d \circ \temprot_d) \left(\mathcal A_1 \psi \right)\rangle 
    = 
    \langle \tau, \psi\rangle. \qedhere
\end{equation*}
\end{proof}

\subsection{Slicing of Radial Regular Tempered Distributions}
\label{sec:distributional_slicing}
In the following, we are mainly interested in tempered distributions of function type, where we skip the word ''tempered'' in the following.
A distribution $T\in \tempd$ is called \emph{regular} if it is generated by a function $T\in L^1_\mathrm{loc}(\Rd)$ via
\begin{equation*}
    \langle T, \varphi\rangle \coloneqq \int_\Rd T(x)\varphi(x)\,\d x\forrall \varphi\in \schd.
\end{equation*}
Regular distributions were characterized in \cite{Szm83}.
Clearly, radial regular distributions arise from radial functions.
A function $T\in L^1_\mathrm{loc}(\Rd)$ is \emph{slowly increasing} if there exist $c,k,R>1$ such that, 
\begin{equation}\label{eq:slow-increase}
|T(x)|\le c\|x\|^k \forrall \|x\|>R,
\end{equation}
see, e.g., \cite{W2004}.
The notion of slow increase is sometimes defined slightly different in the literature, requiring \eqref{eq:slow-increase} to hold on $\R^d$.
Every slowly increasing function generates a regular distribution, but the converse does not hold true.
We are interested in even functions $F$ 
such that $\mathcal R_d F$ is a regular distribution on $\R^d$. Note that this is a weaker assumption than saying that
$F$ itself is an even regular distribution on $\R$.
Then, we can associate to $F$ a distribution
\begin{equation}\label{eq:f_dist}
f \coloneqq
    (\mathcal F_1 \circ \mathcal R_d^\star \circ \mathcal F_d \inv) [ \mathcal R_d F] =  (\mathcal F_1 \circ \mathcal R_d^\star \circ \mathcal F_d \inv \circ \mathcal R_d) F.
\end{equation}
Note that $\mathcal F_d \inv [\mathcal R_d F]$ is in general not a function again.
However, by the following proposition, we will see that \eqref{eq:f_dist} coincides with \eqref{eq:walk_aa} if this is the case.

\begin{proposition} \label{lem: rot distr charac}
\begin{itemize}
 \item[i)]    
 Let $T\in\tempdrad$ be a radial regular tempered distribution. Then  
    \begin{equation}\label{eq:rot distr chrac}
    \langle \temprot_d^\star T , \psi \rangle =
    \tfrac{\omega_{d-1}}{2} \langle  (\mathcal M_d \circ \mathcal A_d )T, \psi \rangle
    \quad \text{for all}  \quad \psi\in \mathcal S(\R).
    \end{equation}
  \item[ii)] 
  For a function $F\colon [0,\infty) \to \R$, let $\mathcal R_d F$ as well as
  $\mathcal F_d^{-1} [\mathcal R_d F]$ be regular distributions.
  Then $f$ in \eqref{eq:f_dist} has the form
  $$
  f=\tfrac{\omega_{d-1}}{2}(\mathcal F_1\circ \mathcal M_d\circ\mathcal A_d\circ \mathcal F_d\inv \circ \mathcal R_d)F.
  $$
\end{itemize}
\end{proposition}

Note that the factor  $\nicefrac{\omega_{d-1}}{2}$ is hidden
in $\mathcal R_d^\star$.

\begin{proof}
i) Firstly, if $T$ is a radial regular distribution, then the same holds true for $(\mathcal M_d \circ \mathcal A_d )T$.
We set $\rho \coloneqq \mathcal A_d T$ which is by definition an even function. Since $T$ is radial, we have by \eqref{eq:inv_ra} that $T=\mathcal R_d\rho$. Further, since  $\mathcal R_d^\star T$ as well as
$\mathcal M_d\rho$ is even,  it suffices by Lemma \ref{lemma invariand under Rad} to reduce ourselves to $\psi \in \schonerad$. 
Then $\mathcal A_1\psi = \psi$ and we obtain
\begin{align*}
    \inn{\temprot_d^\star T,\psi}
    =
    \inn{T, (\mathcal R_d \circ \mathcal A_1)\psi}
    &=
    \int_{\R^d} \rho(\|x\|) \psi(\|x\|)\,  \d x
    =
    \omega_{d-1} \int_{0}^\infty \rho(r) r^{d-1}  \psi(r) \,  \d r
    \\&=
    \frac{\omega_{d-1}}{2} \int_{\R} \rho (r)|r|^{d-1} \psi(r) \, \d r
    =\frac{\omega_{d-1}}{2} \int_{\R} \mathcal M_d\rho (r) \psi(r) \, \d r\\
    &=
    \frac{\omega_{d-1}}{2} \langle \mathcal M_d\rho, \psi\rangle
    =\frac{\omega_{d-1}}{2} \langle \mathcal (\mathcal M_d\circ \mathcal A_d)T, \psi\rangle.
\end{align*}
ii) This part follows directly from  \eqref{eq:f_dist} and part
i) applied to $T= \mathcal F_d^{-1} [\mathcal R_d F]$.
\end{proof}
Under the assumptions of Proposition~\ref{thm:walk_1}, we can apply part ii) of Proposition~\ref{lem: rot distr charac},
therefore we can consider \eqref{eq:f_dist}
as a generalization of \eqref{eq:walk_aa} 
to regular distributions.
The following theorem establishes conditions such that
the pairs of functions $(F,f)$ fulfill
the slicing property \eqref{slicing_def}.
Indeed, it includes functions $F$ for which the Fourier transform of
$\mathcal R_d F$ is not regular,
as the already mentioned function $F(x) = x^r$, $r > -1$, 
from Example~\ref{example}
or positive definite functions $\mathcal R_d F$ having  just a positive measure as Fourier transform.
To prove the theorem, we need the following lemma about a so-called approximate identity,
which can be shown following the lines 
of \cite[Thm.~5.20]{W2004}.

\begin{lemma}\label{lem: approx id}
Let $\Phi\in L^1_\textup{loc}(\mathbb R^d)$ be slowly increasing.
If $\Phi$ is continuous in $z\in \mathbb R^d$, then 
\begin{equation}\label{eq:approx-id2} \Phi(z)=\lim_{m\to \infty} \langle \Phi,\varphi_{d,m,z}\rangle, 
\end{equation}
where $\varphi_{d,m,z}(x)\coloneqq (m/\pi)^{d/2}\e^{-m\|x-z\|^2}$.
\end{lemma}

\begin{theorem}\label{theorem distributional slicing}
    Let $F\colon[0,\infty)\to\R$ such that both $\mathcal R_d F$    and 
    \begin{equation} \label{eq:distributional_slicing}
    f\coloneqq  (\mathcal F_1 \circ \mathcal R_d^\star \circ \mathcal F_d \inv) [ \mathcal R_d F]
    \end{equation}
    are slowly increasing functions. 
    If $f\in \mathcal C(\R)$ and $F$ is continuous at $\|z\|$ for some $z \in \R^d$, then 
    \begin{equation} \label{eq:slicing-regular}
        F(\|z\|)= \E_{\xi\sim\cU_\sphere}[f(|\langle z,\xi\rangle|) ].
    \end{equation}
\end{theorem}

\begin{proof}
For $m\in \N$, we use the above Schwartz function 
$\varphi_{d,m,z}$, which has the Fourier transform
$\hat \varphi_{d,m,z}(v) = \e^{-2\pi\i\langle z,v\rangle} \e^{-\pi^2\|v\|^2/m}$. 
Since $\mathcal R_d F$ is continuous in $\|z\|$ and slowly increasing, we obtain by Lemma~\ref{lem: approx id} 
that
\begin{equation*}
    F(\|z\|)
    =\lim_{m\to \infty} \langle \mathcal R_d F, \varphi_{d, m, z}\rangle
    =\lim_{m\to \infty} \langle \mathcal F_d\inv [\mathcal R_dF], \hat \varphi_{d, m, z}\rangle.
\end{equation*}
By Lemma \ref{Lemma Rot inverse to avg}, we realize that
\begin{equation*}
    (\tempavg_d^*\circ \mathcal F_1\inv )f
     =(\tempavg_d^*\circ \mathcal F_1\inv
     \circ \mathcal F_1\circ \mathcal R_d^\star  \circ \mathcal F_d\inv)[\mathcal R_d F]
    =
    (\tempavg_d^*\circ \mathcal R_d^\star) \circ \mathcal F_d\inv [\mathcal R_d F]
    =
    \mathcal F_d\inv[\mathcal R_d F],   
\end{equation*}
so that
\begin{equation*}
    F(\|z\|)
    =
    \lim_{m\to \infty} \langle (\tempavg_d ^*\circ \mathcal F_1\inv)f, \hat \varphi_{d, m, z}\rangle 
    =\lim_{m\to \infty}  \E_{\xi\sim\cU_\sphere}
    \left[\langle f, \mathcal F_1\inv [(\hat \varphi_{d,m,z})_\xi]\rangle \right],
\end{equation*}
where for any $\xi\in \sphere $ and $r\in\R$,
\begin{equation*}
(\hat \varphi_{d, m, z})_\xi(r)
    \coloneqq\hat \varphi_{d, m, z}(r \xi)
    =\e^{-2\pi\i\langle r\xi,z\rangle}\, \e^{-\pi^2|r|^2/m}
    = \hat \varphi_{1, m,\langle \xi,z\rangle }(r).
\end{equation*}
It follows that
\begin{equation} \label{star}
    F(\|z\|)= \lim_{m\to \infty}  \left[ \E_{\xi\sim\cU_\sphere}  \langle f,  \varphi_{1,m,\langle z,\xi\rangle}\rangle \right].
\end{equation}
Since $f$ is continuous slowly increasing and even, we obtain again by \eqref{eq:approx-id2} that
$$ \lim_{m\to \infty} \langle f,\varphi_{1,m,\langle z,\xi\rangle}\rangle = f(\langle z,\xi\rangle) = f(|\langle z,\xi\rangle|) 
$$
It remains to show that we can interchange the limit and integration in \eqref{star}.
Since $f$ is slowly increasing, there exist $k\in \N$, $c>0$ and $R>0$ such that $|f(r)|\leqslant c|r|^k$ for all $r\geqslant R$.
We choose $R$ large enough so that $r\leqslant 2r^2-k$ for all $r\geqslant R$ and $|f(r)|\leqslant c|r|^k$ for all $r\geqslant R-\|z\|$.\\
For any $r,s\in \R$, the convexity of $|\cdot|^k$ implies that 
\begin{equation}\label{eq:rs-bound}|r+s|^k=|\tfrac{1}{2}(2r)+\tfrac{1}{2}(2s)|^k\leqslant\tfrac{1}{2}|2r|^k+\tfrac{1}{2}|2s|^k=2^{k-1}(|r|^k+|s|^k).\end{equation}
Setting $\varphi_{d,m}\coloneqq\varphi_{d,m,0}$, we split up the integral
\begin{align*}
    |\langle f, \varphi_{1,m,s}\rangle|&
    =|\langle f(\cdot+s), \varphi_{1,m}\rangle|
    \leqslant \int_\R |f(r+s)\varphi_{1,m}(r)|\,\d r\\&
    =\int_{|r|\leqslant R} |f(r+s)|\varphi_{1,m}(r)\,\d r
    +\int_ { |r|>R } |f(r+s)|\varphi_{1,m}(r)\,\d r.   
\end{align*}
Using that $\varphi_{1,m}$ is positive and its integral is one, we estimate the first part
\begin{equation*}
    \int_{|r|\leqslant R} |f(r+s)|\varphi_{1, m}(r)\,\d r
    \leqslant \max_{|r|\leqslant R} |f(r+s)|\int_{|r|\leqslant R} \varphi_{1, m}(r)\,\d r
    \leqslant 
    \max_{|r|\leqslant R} |f(r+s)|.
\end{equation*}
For the second part, we have by \eqref{eq:rs-bound} for all $|s|\leqslant \|z\|$ that
\begin{align*}
    \int_ { |r|>R } |f(r+s)|\varphi_{1,m}(r)\,\d r&
    \leqslant c2^{k-1}\int_{|r|>R} (|r|^k+|s|^k)\varphi_{1,m}(r)\,\d r
    \\&
    \leqslant c2^{k-1}\left(|s|^k + 2\int_R^\infty r^k\varphi_{1,m}(r)\,\d r\right)\\&
    \leqslant c2^{k-1}\left(|s|^k + 2(m/\pi)^{1/2}\int_R^\infty (2mr^{k+1}-kr^{k-1}) \e^{-mr^2}\,\d r\right)
    \\&
    =c2^{k-1}\left(|s|^k+ 2(m/\pi)^{1/2}\left[-r^k\e^{-mr^2}\right]_R^\infty\right)
    \\&
    =c2^{k-1}\left(|s|^k+ 2R^k(m/\pi)^{1/2}\e^{-mR^2}\right).
\end{align*}
Due to the growth of the exponential, we can find $m_0$ such that $2R^k(m/\pi)^{1/2}\e^{-mR^2}\leqslant 1$ for all $m\geqslant m_0$. 
Now let $s=\langle z,\xi\rangle$. Then we have $|s|\leqslant \|z\|$ and for $m\geqslant m_0$ it follows
\begin{align}
   | \langle f,\varphi_{1,m,s}\rangle|&
   \leqslant \max_{|r|\leqslant R} |f(r+s)|+c2^{k-1}(|s|^{ k} + 2R^k\varphi_{1,m}(R))\notag
   \\&\leqslant \max_{|t|\leqslant R+\|z\|} |f(t)|+c2^{k-1}(\|z\|^{ k}+1)\in \mathcal O(\|z\|^k) \text{ as } \|z\|\to \infty.\label{eq:lebesgues_estimate}
\end{align}
This bound is independent of $\xi$ and $m$, therefore we can apply Lebesgue's dominated convergence theorem to \eqref{star} and finally obtain \eqref{eq:slicing-regular}.
\end{proof}

Concerning $\mathcal R_dF$, the conditions in Theorem~\ref{theorem distributional slicing}  are fulfilled if $F$ is continuous on $(0,\infty)$, slowly increasing, and $F(r)r^{d-1}$ is bounded for $r\searrow0$. 
However, the prior knowledge that \(f\) is regular and continuous is required.
It would be more natural to impose conditions solely on the function \(F\).

We will accomplish this by first proving Theorem \ref{thm:dist slicing}, which can be considered as reversed statement of Theorem~\ref{theorem distributional slicing}. 
Afterwards, we use the fractional derivatives of Appendix~\ref{app:sturm_liouv},
in particular Theorem \ref{thm:enusre_slicing},
to show in Corollary \ref{cor: dist slicing F} that the slicing formula \eqref{eq:slicing-regular} 
follows from the distributional slicing formula 
\eqref{eq:distributional_slicing} under certain smoothness assumptions only on the function $F$.

\begin{theorem}\label{thm:dist slicing}
Let $f\in \mathcal C(\R)$ be slowly increasing and even.
Then the radial distribution $\Phi\coloneqq (\mathcal F_d\circ \mathcal A_d^*\circ \mathcal F_1^{-1})[f]$ is regular, continuous, slowly increasing, and satisfies $\Phi(z)=F(\|z\|)$ for all $z\in \R^d$,
where $F$ is defined by \eqref{eq:basis_function}. In particular, $f$ can be recovered from the inversion formula \eqref{eq:distributional_slicing}.
\end{theorem}

The proof, which is similar to Theorem~\ref{theorem distributional slicing}, is given in Appendix~\ref{app:proof-new}

\begin{corollary}\label{cor: dist slicing F}
Let $d\ge 3$ and let the $\lfloor \nicefrac{d}{2}\rfloor$-th derivative of $F\in \mathcal C^{\lfloor \nicefrac{d}{2} \rfloor }([0,\infty))$ be slowly increasing. Then $f\coloneqq (\mathcal F_1^{-1}\circ \mathcal R_d^\star \circ \mathcal F_d)[\mathcal R_d F]$ is continuous and regular and satisfies the slicing relation \eqref{eq:slicing-regular}.
\end{corollary}

\begin{proof}
    Under the given assumptions, Theorem \ref{thm:enusre_slicing} ensures that there is a continuous, slowly increasing function $f\in \mathcal C([0,\infty))$ that satisfies \eqref{eq:slicing-regular}. With $\Phi\coloneqq (\mathcal F_d\circ \mathcal A_d^*\circ \mathcal F_1^{-1})[f]$, Theorem \ref{thm:dist slicing} yields that $\Phi=\mathcal R_dF$ and $f=(\mathcal F_1\circ \mathcal R_d^\star \circ \mathcal F_d^{-1})[\mathcal R_d F]$.  
\end{proof}

The last corollary can be seen as a reverse version of Theorem~\ref{thm:slicing}, which says that for given $f$ of certain integrability, $F$ is $\lfloor(d-2)/2\rfloor$ times differentiable, but Corollary~\ref{cor: dist slicing F} requires slightly more regularity of~$F$.

The following two examples show applications of Theorem~\ref{theorem distributional slicing}.
For the first one, the slicing of the Riesz kernel $F(x)=|x|$ is already known, see Example~\ref{example}. However, the existing proof works in the ``opposite'' direction, i.e., it starts with the sliced kernel $f$ and then computes $F$ via \eqref{eq:basis_function}. In contrast, our method  allows to start with $F$ and compute~$f$.

\begin{example}
For the Riesz kernel $F(x)=|x|$ from Example~\ref{example},
the Fourier transform of $\mathcal R_d F$ does not exist in the classical sense, but as a tempered distribution.
We utilize $\varphi_{d,m}\coloneqq \varphi_{d,m,0}$ from Lemma \ref{lem: approx id}.
Set $f\coloneqq (\mathcal F_1 \circ \mathcal R_d^\star\circ  \mathcal F_d\inv) [\mathcal R_d F]$, see \eqref{eq:f_dist}.
Let $\psi \in \schonerad$.
Since $\mathcal R_d\hat\psi-\hat\psi(0)\hat \varphi_{d,m}\in \mathcal S(\Rd)$ as well as all its first order derivatives vanish at $0$, Wendland \cite[Thm.~8.16]{W2004} yields
 \begin{align}
  \langle f,\psi\rangle &
  =\langle \mathcal F_d\inv[\mathcal R_d F], \mathcal R_d\hat \psi \rangle 
  = \langle \mathcal F_d\inv[\mathcal R_dF], \mathcal R_d\hat \psi-\hat\psi(0)\hat \varphi_{d,m}\rangle +\hat\psi(0)\langle \mathcal F_d\inv [\mathcal R_dF],\hat\varphi_{d,m}\rangle \\&
  = \frac{-2}{\pi \omega_d}\int_\Rd \frac{\hat \psi(\|x\|)-\hat \psi(0)\hat \varphi_{d,m}(x)}{\|x\|^{d+1}}\,\d x+\hat\psi(0)\int_\Rd \|x\|\varphi_{d,m}(x)\,\d x.\label{eq:riesz rep}
  \end{align}  
The limit for $m\to \infty$ of the last term vanishes when we apply Lemma \ref{lem: approx id} to $\mathcal R_d F=\|\cdot\|$ and $z=0$, i.e.,
\begin{equation*}
    \lim_{m\to \infty} \int_\Rd \|x\|\varphi_{d, m}(x)\,\d x = 0.
\end{equation*}
Since $\hat \varphi_{d,m}(x)=\e^{-\pi^2\|x\|^2/m}=\hat \varphi_{1,m}(\|x\|)$, we obtain for the first term
\begin{align*}
    &\quad \int_\Rd \frac{\mathcal R_d\hat\psi(x)-\hat \psi(0)\hat\varphi_{d,m}(x)}{\|x\|^{d+1}}\,\d x 
    =\int_\Rd \frac{\hat\psi(\|x\|)-\hat \psi(0)\hat\varphi_{1,m}(\|x\|)}{\|x\|^{d+1}}\,\d x\\&
    = \omega_{d-1} \int_0^\infty \frac{\hat\psi(r)-\hat\psi(0)\hat \varphi_{1,m}(r)}{r^{d+1}} r^{d-1}\,\d r
    =\frac{\omega_{d-1}}{2}\int_\R \frac{\hat\psi(r)-\hat\psi(0)\hat \varphi_{1,m}(r)}{r^2}\,\d r.
\end{align*}
Then we have
\begin{equation*}
\langle f, \psi\rangle 
    = \frac{-\omega_{d-1}}{\pi \omega_d}\lim_{m\to \infty} \int_\R \frac{\hat \psi(r)-\hat\psi(0)\hat \varphi_{1,m}(r)}{r^2}\,\d r.
\end{equation*}
Employing \eqref{eq:riesz rep} backwards for dimension $1$, we finally obtain
\begin{equation*}
\langle f, \psi\rangle 
    = \frac{\omega_{d-1} \omega_1}{2\omega_d}  \langle \mathcal F_1\mathcal R_1^\star \mathcal F_1\inv [\mathcal R_1F],\psi\rangle 
    = \frac{\pi \omega_{d-1}}{\omega_d}  \int_\R |x| \psi(x) \,\d x.
\end{equation*}
Therefore, it holds
$$f(x)=\frac{\pi \omega_{d-1}}{\omega_d}|x|, \quad x\in\R.$$ 
Both $f$ and $F$ are slowly increasing and continuous functions, so the assumptions of Theorem~\ref{theorem distributional slicing} are satisfied.
\end{example}

The following example adds a new pair $(F,f)$ to the list of sliced functions,
namely
\eqref{newF} and \eqref{pairf}.

\begin{example} \label{ex:Helmholtz}
We consider the fundamental solution, also known as Green's function, of the Helmholtz operator, namely for some $k_0>0$,
\begin{equation} \label{newF} 
F(r) = \frac{\i}{4} \left(\frac{k_0}{2\pi r}\right)^{\frac{d-2}{2}} H^{(1)}_{\frac{d-2}{2}}(k_0r), \forrall r\ge0, 
\end{equation}
where $H^{(1)}_a$ is the Hankel function of the first kind and order $a$. We have 
$$
\inn{\mathcal F_d [\mathcal R_d F],\varphi}
=
\lim_{\varepsilon\searrow0} \int_{\R^d} \frac{\varphi(x)}{4\pi^2\|x\|^2 - k_0^2 - \i\varepsilon} \,\d x,
$$
see \cite{KirQueSet24}.
The significance of this example is that the forward simulation of a scattering problem can be done via the convolution with $\mathcal R_d F$, cf.\ \cite{FauKirQueSchSet23}.
By \cite{KirQueSet24}, $\mathcal R_d F$ is indeed a regular tempered distribution on $\R^d$.
However, the asymptotic form $|H_a(r)|\sim\sqrt{2/(\pi r)}$ for $r\to\infty$, see \cite[\href{https://dlmf.nist.gov/10.2}{10.2.5}]{DLMF}, shows that $\mathcal R_d F\notin L^1(\R^d)$, so we are not in the setting of Proposition~\ref{thm:walk_1}.
Let $\psi\in\schonerad$.
We obtain with transformation to polar coordinates that
\begin{align*}
\inn{\mathcal R_d^\star \mathcal F_d^{-1} [\mathcal R_d F],\psi}
=
\inn{\mathcal F_d [\mathcal R_dF],\mathcal R_d\psi}
&=
\lim_{\varepsilon\searrow0} \int_{\R^d} \frac{\psi(\|x\|)}{4\pi^2\|x\|^2 - k_0^2 - \i\varepsilon} \,\d x
\\&
=
\frac{\omega_{d-1}}{2} \lim_{\varepsilon\searrow0} \int_{\R} \frac{\psi(r) |r|^{d-1}}{4\pi^2 r^2 - k_0^2 - \i\varepsilon} \,\d r.
\end{align*}
By Theorem~\ref{theorem distributional slicing}, we have
\begin{align*}
\inn{f,\psi}
&=
\inn{\mathcal F_1 \mathcal R_d^\star \mathcal F_d^{-1} [\mathcal R_d F],\psi}
=
\inn{\mathcal R_d^\star \mathcal F_d^{-1} \mathcal A_d \mathcal P_d[F],\mathcal F_1[\psi]}
\\&
=
\frac{\omega_{d-1}}{2}
\lim_{\varepsilon\searrow0} \int_{\R} \int_{\R} \frac{ |r|^{d-1}}{4\pi^2 r^2 - k_0^2 - \i\varepsilon} \psi(s) \e^{-2\pi\i rs} \,\d s \,\d r
\\&
=
\omega_{d-1} \lim_{\varepsilon\searrow0} \int_{0}^\infty \int_{\R} \frac{ r^{d-1}}{4\pi^2 r^2 - k_0^2 - \i\varepsilon}  \psi(s) \cos(2\pi rs) \,\d s \,\d r.
\end{align*}
In particular, for $d=2$, we get
\begin{equation*}
\inn{f,\psi}
=
2\pi \lim_{\varepsilon\searrow0} \int_{0}^\infty \int_{\R} \frac{ r}{4\pi^2 r^2 - k_0^2 - \i\varepsilon}  \psi(s) \cos(2\pi rs) \,\d s \,\d r.
\end{equation*}
Noting that 
$$\frac{r}{4\pi^2r^2-k_0^2-\i\varepsilon} = \frac{-r}{k_0^2+\i\varepsilon}\vphantom{F}_1F_0\big(1;-;-\tfrac{4\pi^2}{k_0^2+\i\varepsilon} r^2\big),$$
where $F$ is the hypergeometric function, we have by \cite[8.19(19)]{Erd54b} that
\begin{equation*}
\inn{f,\psi}
=
\frac{1}{4\sqrt\pi} \lim_{\varepsilon\searrow0} \int_{0}^\infty \psi(s) {G^{2,1}_{1,3}\left(-\tfrac14 {(k_0^2+\i\varepsilon)s^2}\middle|\begin{matrix}
    0\\0,0,\frac12
\end{matrix}\right)} \,\d s,
\end{equation*}
where $G$ denotes the Meijer-G function defined by
$$
G_{p,q}^{\,m,n} \!\left( z \middle| \begin{matrix} a_1, \dots, a_p \\ b_1, \dots, b_q \end{matrix} \right) 
\coloneqq \frac{1}{2\pi\i} \int_L \frac{\prod_{j=1}^m \Gamma(b_j - s) \prod_{j=1}^n \Gamma(1 - a_j +s)} {\prod_{j=m+1}^q \Gamma(1 - b_j + s) \prod_{j=n+1}^p \Gamma(a_j - s)} \,z^s \,\d s,
$$
where $L$ is a certain loop in the complex plane,
see \cite[Sect.~9.3]{GrRy07}.
Aside from its poles, the $G$ function has a jump discontinuity along the positive real axis due to taking the main branch of $z^s$.
We obtain
\begin{equation*}
\inn{f,\psi}
=
\frac{1}{4\sqrt\pi} \lim_{\varepsilon\searrow0} \int_{0}^\infty \psi(s) \overline{G^{2,1}_{1,3}\left(\tfrac14 {(-k_0^2+\i\varepsilon)s^2}\middle|\begin{matrix} 0\\0,0,\frac12 \end{matrix}\right)} \,\d s,
\end{equation*}
and the integrand depends continuously on $\varepsilon\ge0$.
Assuming that the limit and the integral can be interchanged, this implies
\begin{equation} \label{pairf}
f(s)
=
\frac{1}{4\sqrt\pi} \overline{G^{2,1}_{1,3}\left(-\tfrac14 {k_0^2s^2}\middle|\begin{matrix} 0\\0,0,\frac12 \end{matrix}\right)}.
\end{equation}
Conversely, we can verify that $f$ indeed fulfills \eqref{eq:basis_function}  using from \cite{Pru3} the integral formula 2.24.2 and the relation 8.4.23.1, in particular for $s>0$,
\begin{align*} 
F(s) 
&= \frac{2}{\pi} \int_0^1 f(ts)(1-t^2)^{\frac{d-3}{2}}\,\d t
\\&= \frac{1}{4\pi^{3/2}} \int_0^1  \overline{G^{2,1}_{1,3}\left(-\tfrac14 {k_0^2us^2}\middle|\begin{array}{l} 0\\0,0,\frac12 \end{array}\right)} (1-u)^{-\frac{1}{2}} u^{-\frac12}\,\d u
\\&= \frac{1}{4\pi} \overline{G^{2,2}_{2,4}\left(-\tfrac14 {k_0^2s^2}\middle|\begin{array}{l} \tfrac12,0\\0,0,\frac12,0 \end{array}\right)} 
= \frac{1}{4\pi} \overline{G^{2,0}_{0,2}\left(-\tfrac14 {k_0^2s^2}\middle|\begin{array}{l} -\\0,0 \end{array}\right)} 
\\&= \frac{1}{2\pi} \overline{K_0(\i k_0s)}
= \frac{\i}{4} H^{(1)}_0(k_0s),
\end{align*}
where $K_0$ is the modified Bessel function of the second kind and the relation of the $G$ functions with different orders follows directly from its definition.
\end{example}

\subsection{Slicing of Positive Definite Functions}
\label{sec:pos-def}
Given a regular distribution $F$, we can always apply the distributional slicing operator \eqref{eq:distributional_slicing} and obtain $f=(\mathcal F_1\circ \mathcal R_d^\star\circ \mathcal F_d^{-1})[\mathcal R_dF]$, which is in general only a tempered distribution. Theorem \ref{theorem distributional slicing} requires $f$ to be a regular continuous distribution, and we do not know simple conditions on $F$ for this to be fulfilled. In this section, we will see that for any function $F$ such that $\mathcal R_d F$ is positive definite, we can ensure via Bochner's theorem that the sliced kernel function $f$ is a regular distribution and indeed also a positive definite function. To this end, we consider measures as a subspace of tempered distributions.

We denote by $\cM(\R^d)$ the set of finite Borel measures on $\R^d$ and by $\cM_+(\R^d)$ the subset of  positive measures. 
The space $\mathcal M(\R^d)$ with the total variation norm $\| \cdot \|_{\textup{TV}}$ is a Banach space. 
Actually, it can be seen as a subspace of $\mathcal S'(\R^d)$ in the following sense, see, e.g.
\cite[Sect.~4.4]{PlPoStTa23}:
by Riesz' representation theorem, it can be identified with the dual space $\mathcal C_0'(\R^d) \cong \mathcal M(\R^d)$
via the isometric isomorphism 
$\mu \mapsto T_\mu$ given by
\begin{equation} \label{eq:measure-dual}
\langle T_\mu,\varphi\rangle \coloneqq 
\int_{\R^d} \varphi \,\d \mu 
\forrall \varphi \in \mathcal C_0(\R^d).
\end{equation}
Since $\mathcal S(\R^d)$ is a dense subspace of $\mathcal (\mathcal C_0(\R^d), \| \cdot \|_\infty)$, every $\mu\in\mathcal M(\R^d)$ can be actually identified with a linear functional  $T_\mu$ on the Schwartz space, which is also continuous with respect to the convergence in $\mathcal S(\R^d)$ by
$$
|\langle T_\mu,\varphi \rangle | \le \| \mu \|_{\text{TV}} \|\varphi\|_\infty 
=\| \mu \|_{\text{TV}} \| \varphi\|_0
\forrall \varphi \in \mathcal S(\R^d)
$$
with $\| \cdot\|_0$ from \eqref{eq:S}. Thus, $T_\mu \in \mathcal S'(\R^d)$, i.e., every measure from $\cM(\R^d)$ corresponds to a
tempered distribution, but not conversely.

The Fourier transform $\mathcal F_d\colon \cM(\R^d) \to \mathcal C_b(\mathbb R^d)$ of measures is an injective, linear transform defined by
\begin{equation}\label{fourier_measure}
\mathcal F_d [ \mu]
\coloneqq
\int_{\R^d} \e^{-2\pi\i\inn{\cdot, v}} \,\d \mu(v).
\end{equation}
Note that $\mathcal F_d [\mu]$  is also known as the characteristic function of $\mu$. 
If $\mu$ is absolutely continuous with respect to the Lebesgue measure with density $\Phi \in L^1(\mathbb R^d)$,
then \eqref{fourier_measure} becomes \eqref{fourier_1}.
If $T_\mu$ is the tempered distribution associated to the measure $\mu$, then it holds for all $\varphi \in \mathcal S(\R^d)$ that
\begin{align*}
 \langle \mathcal F_d T_\mu, \varphi \rangle
 &=
  \langle  T_\mu, \mathcal F_d \varphi \rangle
  = 
  \int_{\R^d}  \mathcal F_d \varphi \, \d \mu 
  =
   \int_{\R^d} \int_{\R^d} \varphi(x) \e^{-2\pi \i \langle x,v \rangle} \, \d x \, \d \mu(v)
   \\&=
    \int_{\R^d} \varphi(x) \, \mathcal F_d[ \mu ](x) \, \d x = \langle 
    T_{\mathcal F_d[ \mu ]}, \varphi \rangle,
\end{align*}
so that
$
\mathcal F_d T_\mu = T_{\mathcal F_d[ \mu ]}
$.

The Fourier transform of positive measures is related with so-called positive definite functions.
A continuous (not necessary radial) function $\Phi\colon \R^d \to \C$ is called \emph{positive definite} 
if for all $N \in \N$, 
all pairwise distinct
$x_j \in \R ^d$, and all $\alpha_j \in \C$, $j=1,\ldots,N$, it holds
$$
\sum_{j=1}^N \sum_{k=1}^N \alpha_j \overline{\alpha_k} \Phi(x_j - x_k) \ge 0.
$$
Positive definite functions are bounded, more precisely $\|\Phi\|_\infty = \Phi(0)$.
Functions $F$ such that the radial functions $\mathcal R_d F$ are positive definite \emph{in every dimension} $d \in \N$
were characterized by Schoenberg via completely monotone functions \cite{Schoenberg1938}.
A well-known example of such a function is the Gaussian function.
Bochner's theorem \cite{Bochner1932} relates the Fourier transform of positive measures with positive definite functions.

\begin{theorem}[Bochner] \label{thm:bochner}
Any positive definite function $\Phi\colon \R^d \to \R$ is the Fourier transform of a positive measure and conversely. 
If $\Phi(0) = 1$, then it is the Fourier transform of a probability measure,
i.e., there exists $\mu \in \mathcal M_+(\R^d)$ such that
$\Phi
= \E_{v \sim \mu} \big[\e^{- 2 \pi \i \langle \cdot,v \rangle} \big].
$
\end{theorem}

Using the above relations of measures and tempered distributions, we obtain the following one-to-one correspondence between positive definite radial functions and their sliced versions.

\begin{theorem}\label{thm:pos def inversion props}
    Let $F\colon[0,\infty)\to\R$  such that $\mathcal R_dF$ is positive definite. Then the function $F$ is $\lfloor\frac{d-2}{2}\rfloor$ times continuously differentiable on $(0,\infty)$.
    Moreover, 
    \begin{equation} \label{eq:slicing-measure} f = \mathcal F_1 \mathcal R_d^\star \mathcal F_d^{-1}[\mathcal R_d F]\end{equation} 
    is positive definite on $\R$ and fulfills the slicing formula \eqref{eq:slice}. 
    Conversely, for every even, positive definite function $f$ on $\R$, the radial function $\mathcal R_dF$ given by \eqref{eq:slice} is positive definite on $\R^d$.
\end{theorem}

\begin{proof}
    Bochner's Theorem \ref{thm:bochner} implies that $\mathcal F_d^{-1}[\mathcal R_d F] \in \cM_+(\R^d)$.
    Since a measure $\mu\in\cM(\R^d)$ is positive if and only if $\inn{\mu,\varphi}\ge0$ for every non-negative function $\varphi\in\schd$, see \cite[Sect.~4.4]{PlPoStTa23}, also $\mathcal R_d^\star \mathcal F_d^{-1}[\mathcal R_d F]$ is a positive measure and again by Bochner's theorem $f$ is a positive definite function on $\R$.
    The slicing identity follows from Theorem~\ref{theorem distributional slicing} by identifying measures with distributions,
    since both $\mathcal R_d F$ and $f$ are continuous and bounded and are therefore slowly increasing. By Theorem \ref{thm:slicing} the function $F$ is $\lfloor\frac{d-2}{2}\rfloor$ times continuously differentiable on $(0,\infty)$.
    Analogously, if $f$ is positive definite on $\R$, then $F$ given by \eqref{eq:slice} is positive definite on $\Rd$.
\end{proof}

The above theorem shows that for any function $F$ such that $\mathcal R_d F$ is positive definite, a sliced kernel function $f$ exists.
A somewhat more intuitive reformulation of the inversion formula \eqref{eq:slicing-measure} for measures is given in the following remark.

\begin{remark}
We state the slicing relation \eqref{eq:slicing-measure} using the pushforward of the norm.
In the setting of Theorem~\ref{thm:pos def inversion props}, recall that $\mu\coloneqq \mathcal F_d^{-1}[\mathcal R_d F]$ is a positive measure on $\R^d$.
The pushforward of a measure $\mu\in\mathcal M_+(\R^d)$
by a Borel map $\iota\colon \R^d\to[0,\infty)$ is defined by
$\iota_\sharp \mu\coloneqq \mu\circ \iota^{-1}$.
For simplicity, we identify any measure $\nu\in\mathcal M(\R^d)$ with the operator $T_\nu\in\mathcal C_0'(\R^d)$ in \eqref{eq:measure-dual}.
Then we have $\langle \| \cdot \|_\sharp\mu, \varphi\rangle = \langle\mu, \varphi\circ \| \cdot \|\rangle = \langle\mu, \mathcal R_d\varphi\rangle$ for all $\varphi\in\mathcal C_0([0,\infty))$.
On the other hand, by the definitions of $\mathcal R_d^\star$ and $\mathcal A_1^*$ in \eqref{eq:Rd-star} and \eqref{eq:Ad-star}, we know that $\langle \mathcal R_d^\star\mu, \psi\rangle = \langle\mu, \mathcal R_d \mathcal A_1\psi\rangle = \langle\mathcal A_1^* \| \cdot \|_\sharp\mu, \psi\rangle$ for all $\psi\in\mathcal C_0(\R)$. Hence, the slicing relation \eqref{eq:slicing-measure} becomes 
$$
f 
= \mathcal F_1 \mathcal A_1^* \| \cdot \|_\sharp \mu 
= \mathcal F_1 \mathcal A_1^* \| \cdot \|_\sharp \mathcal F_d^{-1}[\mathcal R_d F].$$
Here $\mathcal A_1^*$ maps a measure $\nu\in\mathcal M([0,\infty))$ to the even measure $\mathcal A_1^*\nu\in\mathcal M(\R)$, which satisfies $\mathcal A_1^*\nu(B) = \frac12 (\nu(B\cap[0,\infty))+\nu((-B)\cap[0,\infty)))$ for any Borel set $B\subset\R$.
\end{remark}

For odd dimension $d$, \cite[Thm.~7]{Wposdefsmooth} shows that any radial positive definite function $F$ is $\lfloor \frac{d}{2}\rfloor$ times differentiable if $\mathcal R_d F\in L^1(\Rd)$.
In comparison, Theorem \ref{thm:pos def inversion props} gives one derivative less, but only requires $\mathcal R_d F$ being positive definite without further assumptions on $F$.

\section{Conclusions} \label{sec:conclusions}

Slicing is an attractive tool to speed up computations in kernel summations often used in machine learning. We proved inversion formulas to compute the sliced kernel $f$ for a given radial kernel $F$ via a ``dimension walk'' in Fourier space. If $\mathcal R_d F$ is the Fourier transform of an integrable function, Proposition~\ref{thm:walk_1} provides a simple formula consisting of Fourier transforms and a multiplication operator. Our main results, Theorem~\ref{theorem distributional slicing} and Corollary~\ref{cor: dist slicing F}, generalize this to tempered distributions and show that on the distributional level slicing is possible when both $F$ and $f$ are regular and continuous.
If the function $\mathcal R_dF$ is positive definite, Theorem~\ref{thm:pos def inversion props} ensures that the sliced kernel $f$ exists and is a positive definite function that can be computed by \eqref{eq:slicing-measure}.

Along the way, some statements of independent interest were shown. Theorem~\ref{thm:slicing} states that the function $F$ obtained from a given sliced kernel $f$ becomes smoother with increasing dimension $d$. Further, Corollary~\ref{cor:smooth} shows that the Fourier transform of any radial function is $\left\lfloor \nicefrac{(d-2)}{2}\right\rfloor$ times continuously differentiable on $\R^d\setminus\{0\}$.

\section*{Declarations}
\textbf{Acknowledgements:}
We want to thank Robert Beinert and Johannes Hertrich for fruitful discussions.

For open access purposes, the authors have applied a CC BY public copyright license to any author-accepted manuscript version arising from this submission. 

\textbf{Conflict of Interest:} 
On behalf of all authors, the corresponding author states that there is no conflict of interest.

\textbf{Competing Interests:}
The authors have no competing interests to declare that are relevant to the content of this article.

\textbf{Funding Information:}
This research was funded in whole or in part by the German Research Foundation (DFG): STE 571/19-1, project number 495365311, within the Austrian Science Fund (FWF) SFB 10.55776/F68 ``Tomography Across the Scales''. Furthermore, NR  gratefully acknowledges funding by the German Federal Ministry of Education and
Research BMBF 01|S20053B project SA$\ell$E.

\textbf{Author contribution:}
All three authors have contributed equally to the manuscript.

\textbf{Data Availability Statement:}
Data availability is not applicable to this article as no new data were created or analyzed in this study. 

\textbf{Research Involving Human and /or Animals:}
Not applicable because no research involving humans or animals has been conducted for this article.

\textbf{Informed Consent:}
Not applicable because no research involving humans has been conducted for this article.

\bibliographystyle{abbrv}
\bibliography{ref}

\appendix
\section{Proof of Theorem \ref{thm:slicing_b}} \label{app:1}
Let  $x\in\R^d$ with $\|x\|=r$. Denote by $U_x$ an orthogonal matrix such that $U_xx=\|x\|e_1$, 
where $e_1$ is the first unit vector.
Then, it holds
\begin{equation*}
\langle x, \xi \rangle = \langle r U_x^\tT e_1, \xi \rangle =r \langle U_x^\tT e_1, \xi \rangle = r\langle e_1,U_x\xi\rangle
\end{equation*}
and consequently 
\begin{align*}
\E_{\xi\sim \mathcal U_{\Sp^{d-1}}}[f(|\langle x,\xi\rangle|)]
&=
\E_{\xi\sim \mathcal U_{\Sp^{d-1}}}[f(r |\langle e_1,U_x\xi \rangle|)]\\
&=\E_{\xi\sim \mathcal U_{\Sp^{d-1}}}[f(r |\langle e_1,\xi \rangle|)]
=\E_{\xi\sim \mathcal U_{\Sp^{d-1}}}[f(r |\xi_1|)].
\end{align*}
We write $\xi=\xi_1 e_1+\sqrt{1-\xi_1^2}(0,\xi_{2:d})$ with $\xi_1 \in [-1,1]$ and $\xi_{2:d}\coloneqq(\xi_2,...,\xi_d)\in\Sp^{d-2}$. 
Applying \cite[(1.16)]{AtHa12}, which holds for $d \ge 3$, 
we obtain 
\begin{align*}
\E_{\xi\sim \mathcal U_{\Sp^{d-1}}}[f(|\langle x,\xi\rangle|)]
&=
\frac{1}{\omega_{d-1}}\int_{\Sp^{d-1}}f(r|\xi_1|)\,\d\Sp^{d-1}(\xi)\\
&=
\frac{1}{\omega_{d-1}}\int_{-1}^1\int_{\Sp^{d-2}}f(r|\xi_1|)\,\d\Sp^{d-2}(\xi_{2:d})(1-\xi_1^2)^{\frac{d-3}{2}}\,\d\xi_1\\
&=\frac{\omega_{d-2}}{\omega_{d-1}}\int_{-1}^1f(r|t|)(1-t^2)^{\frac{d-3}{2}}\,\d t
=c_d \int_0^1f(rt)(1-t^2)^{\frac{d-3}{2}}\,\d t.
\end{align*}
For $d=2$, the change of variables $\xi=(\cos\theta,\sin\theta)$ with $\theta\in[0,2\pi)$ implies that
$$
\int_{\mathbb S^1} f(r|\xi_1|)\d\mathbb S^{d-1}(\xi) 
= \int_{0}^{2\pi} f(r|\cos\theta|)\d\theta
= 2\int_{-1}^{1} f(r|t|) (1-t^2)^{-\frac12}\d t,
$$
therefore the above reasoning remains valid for $d=2$.\hfill\qedsymbol{\parfillskip0pt\par}

\section{Proof of Theorem \ref{thm:slicing}} \label{app:walk_a}

To prove Theorem \ref{thm:slicing} we need some technical lemmas.

\begin{lemma}\label{lemma:tech}
    For all $0<h<s$ and $0<t<s-h$, it holds
    \begin{equation*}
        \frac{1}{h}\left(\Big(1-\frac{t^2}{s^2}\Big)^\frac{1}{2}-\Big(1-\frac{t^2}{(s-h)^2}\Big)^\frac{1}{2}\right)
		  < \frac{1}{s-t}\Big(1-\frac{t^2}{s^2}\Big)^\frac{1}{2}.
    \end{equation*}
\end{lemma}

\begin{proof}
    Since $h<s$ and $s>0$, we have $-3s+h=-2s-(s-h)<0$ and therefore
    \begin{equation*}
        t
        <s-h 
        =\frac{s-h}{2s-h}(2s-h)
        =\frac{2s^2-3sh+h^2}{2s-h}
        < \frac{2s^2}{2s-h}.
    \end{equation*}
We can multiply the inequality with $t(2s-h)$ and obtain
\begin{align*}
   0
   <  2s^2t-t^2(2s-h)
    &=t\left((s-h)^2+s^2\right)+th(2s-h)-t^2(2s-h)
	\\&
   <t\left((s-h)^2+s^2\right)+(s-h)s(2s-h)-t^2(2s-h).
\end{align*}
Multiplying the inequality with $h>0$, a straightforward calculation yields
\begin{align*}
    0&
    < ht\left((s-h)^2+s^2\right)+(s-h)s(2s-h)h -h(2s-h)t^2
		\\&
    =(s-h+t)(s-t)s^2-(s-h-t)(s+t)(s-h)^2.
\end{align*}
Further, since $s-t$, $s^2$, and $(s-h)^2$ are positive, we obtain  
\begin{align*}
    0&
    <\frac{(s-h)+t)}{(s-h)^2}-\frac{((s-h)-t)}{s-t}\, \frac{s+t}{s^2}
		\\&
    =\frac{1}{s-h-t}\Big(\frac{(s-h+t)(s-h-t)}{(s-h)^2}-\frac{(s-h-t)^2}{(s-t)^2}\frac{(s^2-t^2)}{s^2}\Big)
		\\&
    =\frac{1}{s-h-t}\left( \left(1-\frac{t^2}{(s-h)^2}\right) -\left(1-\frac{h}{s-t}\right)^2\left(1-\frac{t^2}{s^2}\right)\right).
\end{align*}
Multiplying with $s-h-t>0$, reordering and taking the square root yields
\begin{equation*}
    \Big(1-\frac{h}{s-t}\Big)\Big(1-\frac{t^2}{s^2}\Big)^\frac{1}{2}
	<\Big(1-\frac{t^2}{(s-h)^2}\Big)^\frac{1}{2}.
\end{equation*}
Finally, we rearrange the equation and divide by $h>0$ to get the assertion
\begin{equation*}
    \frac{1}{h}\left(\Big(1-\frac{t^2}{s^2}\Big)^\frac{1}{2}
		-\Big(1-\frac{t^2}{(s-h)^2}\Big)^\frac{1}{2}\right)
		< \frac{1}{s-t}\Big(1-\frac{t^2}{s^2}\Big)^\frac{1}{2}.\qedhere
\end{equation*}
\end{proof}

\begin{lemma}\label{lem:deriv} 
Let $f\in \lloc$ if $\nu \geqslant 1$ and $f\in L^p_\textup{loc}([0,\infty))$ with $p>2$ if $\nu= \nicefrac12$. 
For $s>0$, we define 
\begin{equation} \label{Inu}
    I_\nu f(s) \coloneqq \int_0^s f(t) \Big(1-\frac{t^2}{s^2}\Big)^\nu \, \d t,
\end{equation}
then it holds that
\begin{equation} \label{deriv_old}
    \frac{\d}{\d s} I_\nu f(s)=\frac{2\nu}{s^3}I_{\nu-1} g(s), \qquad g(t) \coloneqq f(t) t^2.
\end{equation}
\end{lemma}

\begin{proof}
We show that
$$
\lim_{h \to 0} \frac{I_\nu f(s+h) - I_\nu f(s)}{h} = \frac{2\nu}{s^3}I_{\nu-1}g(s).
$$
1. First, we consider $h >0$, i.e., the right-sided limit
\begin{align} \label{up}
     \lim_{h \searrow 0} \frac{I_\nu f(s+h)-I_\nu f(s)}{h}
		&=
    \lim_{h\searrow 0} \frac{1}{h}\int _s^{s+h} f(t)\Big(1-\frac{t^2}{(s+h)^2}\Big)^\nu \, \d t 
		\\
		& \quad +\lim_{h\searrow 0} \int_0^sf(t)\frac{1}{h}\left(\Big(1-\frac{t^2}{(s+h)^2}\Big)^\nu-\Big(1-\frac{t^2}{s^2}\Big)^\nu\right) \, \d t.
\end{align}	
We show that the first summand is zero, while the second one equals ${2\nu}s^{-3}I_{\nu-1}g(s)$.

1.1. Concerning the first limit, we have for $h>0$ and $s<t<s+h$ that$$
\left|1-\frac{t^2}{(s+h)^2}\right|
\leqslant 1-\frac{s^2}{(s+h)^2}
=h\frac{2s+h}{(s+h)^2},
$$
and further by the monotony of the power function, for $0<h<1$, that
\begin{align*}
    \left|\frac{1}{h}\int_s^{s+h} f(t) \left(1-\tfrac{t^2}{(s+h)^2}\right)^\nu\,\d t \right| 
    &
    \leqslant 
		\frac{1}{h}\int_s^{s+h} |f(t)|\, \left(h\frac{2s+h}{(s+h)^2}\right)^\nu\,\d t
	\\
	&
    \leqslant 
		\left( \frac{2s+1}{s^2}\right)^\nu h^{\nu-1} \int_s^{s+h} |f(t)| \, \d t .
\end{align*}	
If $\nu\geqslant 1$ and $f\in \lloc$, then $\nu-1\geqslant 0$ and we  can estimate 
\begin{equation*}
    \left(\frac{2s+1}{s^2}\right)^\nu h^{\nu-1} \int_s^{s+h} |f(t)| \, \d t 
		\leqslant 
		\left(\frac{2s+1}{s^2}\right)^\nu  \int_s^{s+h} |f(t)| \, \d t\xrightarrow{h\searrow 0} 0.
\end{equation*}
If $\nu=\nicefrac12$, we assumed that $f \in L^{p}_{\loc}([0,\infty))$ with $p>2$ and thus $f\in L^{2}_{\loc}([0,\infty))\subset L^{p}_{\loc}([0,\infty))$. Then we get by 
H\"older's inequality with 
\begin{align*}
\left(\frac{2s+1}{s^2}\right)^\nu h^{\nu-1} \int_s^{s+h} |f(t)|\, \d t 
&\leqslant \left(\frac{2s+1}{s^2}\right)^\frac{1}{2} h^{-\frac{1}{2}} 
\left(\int_s^{s+h} |f(t)|^2 \, \d t\right)^\frac{1}{2}\left(\int_s^{s+h}  \,\d t\right)^\frac{1}{2}\\&
=\left(\frac{2s+1}{s^2}\right)^\frac{1}{2}\left( \int_s^{s+h} |f(t)|^2 \, \d t \right)^\frac{1}{2}\xrightarrow{h\searrow 0} 0.
\end{align*}

1.2. Concerning the second limit in \eqref{up}, we use that 
\begin{equation}
    \frac{\d}{\d s}\Big( 1-\frac{t^2}{s^2}\Big)^\nu 
		=
    \frac{ 2t^{2} \nu}{ s^{3}} \Big(1 - \frac{t^{2}}{s^{2}} \Big)^{\nu - 1}.
\end{equation}
If $\nu\geqslant 1$ and $f\in \lloc$, then the mean value theorem implies
\begin{align*}
    &\quad h\inv \left| \Big(1-\tfrac{t^2}{(s+h)^2}\Big)^\nu-\Big(1-\tfrac{t^2}{s^2}\Big)^\nu\right|
    \leqslant 
    \sup_{\xi\in (0,h)} \left|2t^2\nu\Big( 1-\frac{t^2}{(s+\xi)^2}\Big)^{\nu-1}\frac{1}{(s+\xi)^3}\right|
    \\&
    \leqslant \sup_{\xi \in (0,h)} 2t^2\nu\Big( 1+\frac{t^2}{(s+\xi)^2}\Big)^{\nu-1}\frac{1}{(s+\xi)^3}
    \leqslant  2t^2\nu\Big( 1+\frac{s^2}{s^2}\Big)^{\nu-1}\frac{1}{s^3}
    = \nu2^\nu \frac{t^2}{s^3}
    \leqslant \nu 2^\nu s\inv .
\end{align*}
Hence $t\mapsto f(t)\nu 2^\nu s\inv$ is an $h$ independent integrable majorant and 
Lebesgue's dominated convergence theorem gives
\begin{equation}\label{eq:int_above}
    \lim_{h\searrow 0}\int_0^{s} f(t)\frac{1}{h}\left(\left(1-\tfrac{t^2}{(s+h)^2}\right)^\nu-
    \left(1-\tfrac{t^2}{s^2}\right)^\nu\right) \, \d t 
    = 
    \frac{2\nu}{s^3}\int_0^s f(t)t^2 \left(1-\tfrac{t^2}{s^2}\right)^{\nu-1} \, \d t.
\end{equation}
If $\nu = \nicefrac{1}{2}$ and $f\in L^{p}_{\loc}([0,\infty))$ with $p>2$, then 
\begin{align*}
    \frac1h \left| \left(1-\tfrac{t^2}{(s+h)^2}\right)^{\frac12}-\Big(1-\tfrac{t^2}{s^2}\Big)^{\frac12}\right|
    &\leqslant 
	\sup_{\xi\in (0,h)} \left|\frac{t^2}{(s+\xi)^3} \Big( 1-\frac{t^2}{(s+\xi)^2}\Big)^{-\frac12}\right|
	&
    \leqslant 
    \frac{t^2}{s^3} \Big( 1-\frac{t^2}{s^2}\Big)^{-\frac{1}{2}} .
\end{align*}
Denote by $q = (1-p^{-1})^{-1}$ the Hölder conjugate to $p$. Since $p>2$ it holds $q<2$ and also 
\begin{equation*}
    \Big( 1-\frac{t^2}{s^2}\Big)^{-\frac{1}{2}q}
    =\left(\frac{s+t}{s^2}\right)^{-\frac{q}{2}} (s-t)^{-\frac{q}{2}}.
\end{equation*}
For $t\in (0,s)$, both ${t^2}s^{-3}$ and $((s+t) s^{-2})^{-{q}/{2}}$ are bounded, and $(s-t)^{-{q}/{2}}$ is integrable on $[0,s]$. Therefore, $t\mapsto {t^2}s^{-3} ((s+t) s^{-2})^{-{q}/{2}} (s-t)^{-q/2} f(t)$ is an integrable majorant.
Hence, Lebesgue's dominated convergence theorem yields \eqref{eq:int_above}.

2. Next, we deal with the left-sided limit
\begin{align*}
    \lim_{h\nearrow 0} \frac{I_\nu f(s+h)-I_\nu f(s)}{h}
    &=\lim_{h\searrow 0} \frac{I_\nu f(s)-I_\nu f(s-h)}{h}\\&
    =\lim_{h\searrow 0}\int_0^{s} \frac{f(t)}{h}
	\left(\left(1-\tfrac{t^2}{s^2}\right)^\nu-\left(1-\tfrac{t^2}{(s-h)^2}\right)^\nu\one_{[0,s-h]}(t)\right)\, \d t.
\end{align*}	
Define the functions $m_h,m\colon (0,s)\to \R$ by 
\begin{equation*}
m_h(t)\coloneqq \frac{1}{h}\left(\left(1-\tfrac{t^2}{s^2}\right)^\nu-\left(1-\tfrac{t^2}{(s-h)^2}\right)^\nu\one_{[0,s-h]}(t)\right)
\quad \text{ and } \quad
m(t)\coloneqq \frac{\d}{\d s}\left(1-\frac{t^2}{s^2}\right)^\nu.
\end{equation*}
Note that for every $t\in (0,s)$ we have $\lim_{h\to 0} m_h(t)=m(t)$.

2.1. Let $0<t<s-h$. We choose $0<h<\nicefrac{s}{2}$.
If $\nu\geqslant 1$ and $f\in \lloc$, then the mean value theorem implies
\begin{align*}
    |m_h(t)|
	&= \frac1h \left| \Big(1-\tfrac{t^2}{(s-h)^2}\Big)^\nu-\Big(1-\tfrac{t^2}{s^2}\Big)^\nu\right|
    \leqslant \sup_{\xi\in (0,h)} \left|2\nu\Big( 1-\frac{t^2}{(s-\xi)^2}\Big)^{\nu-1}\frac{t^2}{(s-\xi)^3}\right|
		\\&
    \leqslant \sup_{\xi \in (0,h)}2\nu\Big( 1+\frac{t^2}{(s-\xi)^2}\Big)^{\nu-1}\frac{t^2}{(s-\nicefrac{s}{2})^3}
    \leqslant 2\nu\Big( 1+\frac{s^2}{(s-\nicefrac{s}{2})^2}\Big)^{\nu-1}\frac{t^2}{(\nicefrac{s}{2})^3}
		\\&
    \leqslant 16\nu \frac{t^2}{s^3}5^{\nu-1}
    \leqslant 16\cdot 5^{\nu-1} \nu s\inv ,
\end{align*}
which is bounded independently of $t$ and $s$.
Hence, 
Lebesgue's dominated convergence theorem implies 
$$\lim_{h\nearrow 0} \frac{I_\nu(f)(s+h)-I_\nu(s)}{h} 
= \frac{2\nu}{s^3}\int_0^s f(t)t^2 \left(1-\tfrac{t^2}{s^2}\right)^{\nu-1} \, \d t.$$
For $\nu=\nicefrac{1}{2}$ and $f\in L^p_\textup{loc}([0,\infty))$ with $p>2$ and Hölder conjugate $q<2$ to $p$, we apply Lemma \ref{lemma:tech} to obtain
\begin{equation} \label{majo}
|m_h(t)|
\leqslant
   w(t) \coloneqq \frac{1}{s-t}\Big(1-\frac{t^2}{s^2}\Big)^{\frac12}
   = \frac1s \Big( \frac{s+t}{s-t}\Big)^\frac{1}{2},
\end{equation}
which is $q$ integrable.
The claim follows again by using the Hölder inequality and the dominated convergence.
2.2. Let $s-h\leqslant t<s$ and $0<h<s$. For $\nu\geqslant1$, we see that $m_h$ is integrable.
For $\nu=1/2$, we have $s-t\leqslant h$ and consequently 
\begin{equation*}
    |m_h(t)|
	=
	\frac{1}{h}\Big(1-\frac{t^2}{s^2}\Big)^\frac{1}{2}
	\leqslant 
	\frac{1}{s-t}\Big(1-\frac{t^2}{s^2}\Big)^\frac{1}{2} = w(t),
\end{equation*}
see \eqref{majo}.
Finally, we use Lebesgue's dominated convergence theorem again.
\end{proof}

\noindent
\textbf{Proof of Theorem \ref{thm:slicing}}.
Since $s\mapsto s^{-1}$ is smooth on $(0,\infty)$, the differentiability of $F$ in \eqref{eq:basis_function}
depends only on the integral $I_{\nicefrac{(d-3)}{2}} f$ defined in \eqref{Inu}.
By Lemma \ref{lem:deriv}, we have 
for $f\in \lloc$ if $d \ge 5$ as well as for $f\in L^p_\textup{loc}([0,\infty))$ with $p>2$ if $\nu = \nicefrac12$ and 
 $g_k(t) \coloneqq f(t)t^{2k}$
that
\begin{equation} \label{deriv}
    \frac{\d}{\d s} I_\nu f (s)=\frac{2\nu}{s^3} I_{\nu-1} g_1 (s).
\end{equation}
Set $\nu \coloneqq \nicefrac{(d-3)}{2}$. We show that for each $n=0,\ldots,\lfloor \nu \rfloor$, 
there exist smooth functions $r_0,\ldots, r_n$ on $(0,\infty)$ such that
\begin{equation}\label{eq:hderiv}
    \frac{\d^n}{\d s^n} F(s)
    =
    \frac{\d^n}{\d s^n} \left(\frac{c_d}{s}I_{\nu}(s)\right) 
    =\sum_{k=0}^n r_k(s) I_{\nu-k} g_k(s).
\end{equation}
For $n=0$, we obtain that $r_0 = \nicefrac{c_d}{s}$.
Assume the assertion holds for $n < \lfloor \nu \rfloor$. 
For $k=0,\ldots, n$, we have $g_k \in L^p_\textup{loc}([0,\infty))$ with $p>2$ if $\nu=\nicefrac{1}{2}$ and $p=1$ if $\nu\geqslant 1$.
Also $\nu-k\geqslant \lfloor \nu \rfloor-n\geqslant 1$, so that
\eqref{deriv} yields
\begin{align}
 \frac{\d^{n+1}}{\d s^{n+1}} F(s)
  &=\sum_{k=0}^n \left( r_k'(s) I_{\nu-k}g_k (s) + r_k(s) \frac{2\nu}{s^3}I_{\nu-k-1}g_{k+1}(s) \right)\\
 &=\sum_{k=0}^{n+1} \left( r_k'(s) + r_{k-1} (s) \frac{2\nu}{s^3} \right) I_{\nu-k}g_k(s) , \quad r_{n+1}' = r_{-1} \coloneqq 0. 
\label{deriv+}
\end{align}
Hence $F$ is $\lfloor \nu\rfloor$ times differentiable. Moreover, the parameter integrals $I_k$, $k=0,\ldots,\lfloor \nu\rfloor$ 
are  absolutely continuous, which follows from Lemma~\ref{lem:deriv} for $k\ge1$ and from the definition of $I_0$. Hence, also the $\lfloor \nu \rfloor$-th derivative of $F$ is absolutely continuous.

If $d$ is odd, it holds $\lfloor \nu \rfloor = \lfloor \nicefrac{(d-2)}{2} \rfloor$ and we are done.
If $d$ is even, then $\lfloor \nu \rfloor = \nicefrac{(d-4)}{2}$ and $\nu - \lfloor \nu \rfloor = \nicefrac{1}{2}$.
For $f \in L_\textup{loc}^p([0,\infty))$ and $p>2$, we obtain by \eqref{deriv+} and Lemma~\ref{lem:deriv} that the $\lfloor \nu \rfloor+1$ th derivative of $F$
also exits and is absolutely continuous.

To show that the result is tight, consider the function $f(t)\coloneqq 0$ for $t\in[0,1]$ and $f(t)\coloneqq -2t$ for $t\in(1,\infty)$. Then $f\in L^\infty_\textup{loc}([0,\infty))$. For $s\le 1$, we get $F(s)=0$. For $s>1$, we have
\begin{align*}
 \frac{F(s)}{c_d}&
 =\int_0^1f(ts)(1-t^2)^\frac{d-3}{2}\d t
 =\int_{s^{-1}}^1 -2ts (1-t^2)^\frac{d-3}{2}\d t
 =\left[ \frac{2s}{d-1}(1-t^2)^{\frac{d-1}{2}}\right]_{s^{-1}}^1\\&
 = -\frac{2 s}{d-1}\left(1-\tfrac{1}{s^2}\right)^\frac{d-1}{2}
 =- \frac{2s}{d-1}\left(\tfrac{s+1}{s^2}\right)^\frac{d-1}{2} (s-1)^\frac{d-1}{2}\eqqcolon   g(s)(s-1)^\frac{d-1}{2},
\end{align*}
where the function $g$ is analytic and positive around $1$.
If $d$ is odd and ${m}=\lfloor\nicefrac{d}{2}\rfloor=\nicefrac{(d-1)}{2}$, then we obtain by Cauchy's integral formula that 
\begin{equation*}
\left.\frac{\d^{m}}{\d s^{m}}(g(s)(s-1)^{m})\right|_{s=1}=\frac{{m}!}{2\pi \i}\oint\limits_{|z-1|=1} \frac{g(z)(z-1)^{m}}{(z-1)^{{m}+1}} \, \d z
=\frac{{m}!}{2\pi \i}\oint\limits_{|z-1|=1}
\frac{g(z)}{z-1} \, \d z
={m}! g(1)
\end{equation*}
is non-zero.
Hence, we conclude
$$\lim_{s\searrow 1} F^{(m)}(s)=c_d {m}!g(1)\neq 0 = \lim_{s\nearrow 1} F^{(m)}(s).$$
Therefore, $F$ is not ${m}$ times continuously differentiable at $1$.

If $d$ is even, we assume by contradiction that $F$ is $\lfloor \nicefrac{d}{2}\rfloor=\nicefrac{d}{2}$ times continuously differentiable at $1$. Then the function 
$ \nicefrac{F(s)}{g(s)}$
would be $\nicefrac{d}{2}$ times continuously differentiable at $1$, since $g$ is analytic and positive at $1$. We obtain a contradiction as
$$ (s-1)^\frac{d-1}{2} = \frac{F(s)}{c_dg(s)}$$
is not $\nicefrac{d}2$ times differentiable in $1$.
This finishes the proof. \hfill \qedsymbol{\parfillskip0pt\par}

\section{Proof of Theorem \ref{prop def of avg}} \label{app:avg}
To prove the theorem, we need several auxiliary lemmata.

\begin{lemma}\label{Lemma: V operator continous}
For $\psi\in \schone$ with $\psi(0)=0$, the function 
\begin{equation}\label{eq:V}
V[\psi](x)\coloneqq 
\begin{cases}
\psi(x)/x &\text{ if }x\neq 0,\\
\psi'(0) &\text{ if }x= 0.
\end{cases}
\end{equation} 
is in $\schone$ and the respective map $V\colon \{\psi\in \schone\mid \psi(0)=0\}\to \schone$ is continuous with $\|V[\psi]\|_n\leqslant a_n\|\psi\|_{n+1}$, where
\begin{equation}\label{eq:an}
    a_n\coloneqq \max_{m=0,\ldots, n} \left\{\max\left\{ \sum_{l=0}^m \frac{2^n}{m-l+1}\binom{m}{l}, \sum_{l=0}^m \frac{m!}{l!}\right\}\right\}\geqslant 1.
\end{equation}
\end{lemma}

\begin{proof}
Since $\psi(0)=0$, we see that $V[\psi]$ is continuous.
We note that if a function $\phi\in \cc(\R)$ is differentiable on $\R\setminus\{0\}$ and $ \lim_{x\to 0} \phi'(x)$ exists, then it is differentiable in~$0$.
\emph{Part 1:} We show by induction that for $m\geqslant 1$ the $m$-th derivative of $V[\psi]$ is continuous and can be represented as  
\begin{equation}\label{eq:m th derivative V}
    \frac{\dd^m}{\dd x^m} V[\psi](x)=x^{-m-1} \sum_{l=0}^m \psi^{(l)}(x)x^l (-1)^{l+m}\frac{m!}{l!} \forrall x\in \R\setminus\{0\}.
\end{equation}

For $m=1$ we have for $x\neq 0$
\begin{equation*}
    \frac{\dd}{\dd x} V[\psi](x)
    =\frac{\dd}{\dd x}\frac{\psi(x)}{x}
    =\frac{\psi'(x)x-\psi(x)}{x^2}
    =x^{-m- 1} \sum_{l=0}^1 \psi^{(l)}(x)x^l(-1)^{l+1}\frac{1!}{l!}.
\end{equation*}
By L'Hospital it holds
\begin{align*}
    \lim_{x\to 0} \frac{\dd}{\dd x}V[\psi](x)
    =\lim_{x\to 0} \frac{\psi'(x)x-\psi(x)}{x^2}
    =\lim_{x\to 0}\frac{\psi''(x)x}{2x}
    =\frac{\psi^{(2)}(0)}{2}.
\end{align*}
Since $V[\psi]$ is continuous on $\R$ and $\lim_{x\to 0} V[\psi]'(x)$ is finite, we have $V[\psi]'(0)=\psi^{(2)}(0)/2$. In particular, the derivative is continuous.

Now assume, that \eqref{eq:m th derivative V} holds for all derivatives less or equal $m$ and all derivatives less or equal $m$ are continuous.
We can compute the derivative 
\begin{align*}
&\quad \frac{\dd^{m+1}}{\dd x^{m+1}}V[\psi](x)
=\frac{\dd}{\dd x} \left(x^{-m-1} \sum_{l=0}^m \psi^{(l)}(x)x^l (-1)^{l+m}\frac{m!}{l!}\right)\\&
=x^{-m-1}\frac{\dd}{\dd x}\left( \sum_{l=0}^m \psi^{(l)}(x)x^l (-1)^{l+m}\frac{m!}{l!}\right)
- (m+1)x^{-m-2} \sum_{l=0}^m \psi^{(l)}(x)x^l (-1)^{l+m}\frac{m!}{l!}.
\end{align*}
With the calculation
\begin{align}
   &\quad  \frac{\dd}{\dd x}\sum_{l=0}^m \psi^{(l)}(x)x^l (-1)^{l+m}\frac{m!}{l!}
   \\&
   =\sum_{l=0}^m \psi^{(l+1)}(x)x^l(-1)^{l+m}\frac{m!}{l!}+\sum_{l=1}^m \psi^{(l)}(x)x^{l-1}(-1)^{l+m}\frac{m!}{(l-1)!}\notag\\&
   =\sum_{l=0}^m \psi^{(l+1)}(x)x^l(-1)^{l+m}\frac{m!}{l!}+\sum_{l=0}^{m-1} \psi^{(l+1)}(x)x^{l}(-1)^{l+1+m}\frac{m!}{l!}\notag\\&
   =\psi^{(m+1)}(x)x^m(-1)^{2m}\frac{m!}{m!}
   =\psi^{(m+1)}(x)x^m ,\label{eq: derivative of sum f l+1 xl}
\end{align}
we obtain that
\begin{align*}
\frac{\dd^{m+1}}{\dd x^{m+1}}V[\psi](x)
&=x^{-m-1} \, \psi^{(m+1)}(x) x^{m} -x^{-m-2} \sum_{l=0}^m \psi^{(l)}(x)x^l (-1)^{l+m}\frac{(m+1)!}{l!}\\&
=x^{-m-2} \sum_{l=0}^{m+1} \psi^{(l)}(x)x^l (-1)^{l+(m+1)}\frac{(m+1)!}{l!},
\end{align*}
which shows \eqref{eq:m th derivative V}.
Since $\psi$ is a Schwartz function, it is clear that $\frac{\dd^{m+1}}{\dd x^{m+1}}V[\psi](x)$ is continuous on $\R\setminus\{0\}$. Now we use again L'Hospital and \eqref{eq: derivative of sum f l+1 xl} to obtain
\begin{align*}
&\quad \lim_{x\to 0} \frac{\dd^{m+1}}{\dd x^{m+1}}V[\psi](x)
= \lim_{x\to 0} \frac{ \sum_{l=0}^{m+1} \psi^{(l)}(x)x^l (-1)^{l+(m+1)}\frac{(m+1)!}{l!}}{x^{m+2}}\\&
=\lim_{x\to 0} \frac{ \frac{\dd}{\dd x}\sum_{l=0}^{m+1} \psi^{(l)}(x)x^l (-1)^{l+(m+1)}\frac{(m+1)!}{l!}}{\frac{\dd}{\dd x}x^{(m+1)+1}}
=\lim_{x\to 0} \frac{\psi^{(m+2)}(x)x^{m+1}}{(m+1)x^{m+1}}
=\frac{\psi^{(m+2)}(0)}{m+1}.
\end{align*}
This yields that $V[\psi]\in \cc^{m+1}(\R)$ and this induction is finished.
\emph{Part 2:}
Next we show that $V[\psi]$ is a Schwartz function and that $V$ is continuous. Let $x\in\R$ and $m\in\N$ be fixed. We use the Taylor expansion with the Lagrange reminder,
\begin{equation*}
    \psi^{(l)}(x)=\sum_{k=l}^{m} \frac{\psi^{(k)}(0)}{(k-l)!}x^{k-l} +\frac{\psi^{(m+1)}(\xi_{l}(x))}{(m-l+1)!}x^{m-l+1}  \quad \text{for some} \quad |\xi_l(x)|<|x|.
\end{equation*}
With the representation \eqref{eq:m th derivative V}, we see that
\begin{align*}
    &\quad x^{m+1}\frac{\dd^m}{\dd x^m}V[\psi](x)
    = \sum_{l=0}^m \psi^{(l)}(x)x^l (-1)^{l+m}\frac{m!}{l!}\\&
    =x^{m+1}\sum_{l=0}^m \frac{\psi^{(m+1)}(\xi_l(x))}{(m-l+1)!}(-1)^{l+m}\frac{m!}{l!} +\sum_{l=0}^m \sum_{k=l}^{m} \frac{\psi^{(k)}(0)}{(k-l)!}x^{k}(-1)^{l+m}\frac{m!}{l!}\\&
    =x^{m+1} \sum_{l=0}^m \frac{\psi^{(m+1)}(\xi_l(x))}{m-l+1}(-1)^{l+m} \binom{m}{l}   +   \sum_{k=0}^m x^{k}\psi^{(k)}(0)\sum_{l=0}^k \frac{(-1)^{l+m}}{(k-l)!}\frac{m!}{l!}.
\end{align*}
If $k\geqslant 1$, we have by the binomial theorem
\begin{equation*}
    \sum_{l=0}^k \frac{(-1)^{l+m}}{(k-l)!}\frac{m!}{l!}
    =(-1)^m \frac{m!}{k!}\sum_{l=0}^k (-1)^l \binom{k}{l}
    =0.
\end{equation*}
Since $\psi(0)=0$, only the first sum remains and we have 
\begin{equation}\label{eq: dm/dxm V expension with lagrange reminder}
\frac{\dd^m}{\dd x^m}V[\psi](x) =\sum_{l=0}^m \psi^{(m+1)}(\xi_l(x)) \frac{(-1)^{l+m}}{m-l+1} \binom{m}{l} \quad \text{ with } \quad |\xi_l(x)|\leqslant |x|.
\end{equation}
Fixing some natural number $m\leqslant n$, we have for $x\in [-1,1] \setminus\{0\}$ using \eqref{eq: dm/dxm V expension with lagrange reminder}, that
\begin{align*}
   &\quad  \left|(1+|x|)^n \frac{\dd^m}{\dd x^m}V[\psi](x)\right|
   =\left| (1+|x|)^n \sum_{l=0}^m \psi^{(m+1)}(\xi_l(x)) \frac{(-1)^{l+m}}{m-l+1} \binom{m}{l} \right|\\&
    \leqslant (1+|x|)^n \sup_{y\in [-1,1]} \psi^{(m+1)}(y) \sum_{l=0}^m\frac{1}{m-l+1}\binom{m}{l}
    \\&
    \leqslant 2^n \|\psi\|_{n+1} \sum_{l=0}^m \frac{1}{m-l+1}\binom{m}{l}\ \leqslant a_n.
\end{align*}
On the other hand, if $|x|>1$, we have by \eqref{eq:m th derivative V}
\begin{equation*}
 \left|(1+|x|)^n \frac{\dd^m}{\dd x^m}V[\psi](x)\right|
 \leqslant (1+|x|)^{n} \sum_{l=0}^m \left|\psi^{(l)}(x)\right|\frac{m!}{l!}
 \leqslant a_n\|\psi\|_{n+1}.
\end{equation*}
Since $V[\psi]$ is smooth all its derives are continuous, we have
\begin{equation*}
    \|V[\psi]\|_n
    =\sup_{m=0,\ldots, n} \sup_{x\in \R\setminus\{0\}}  \left|(1+|x|)^n \frac{\dd^m}{\dd x^m}V[\psi](x)\right|
    \leqslant a_n\|\psi\|_{n+1}<\infty.
\end{equation*}
Thus $V[\psi]$ is a Schwartz function and the operator $V$ is continuous.
\end{proof}

\begin{lemma}\label{lemma: W operator continuous}
  The operator 
  \begin{equation} \label{eq:W}
    W\colon \schonerad\to \schonerad,\quad W\coloneqq V\circ \tfrac{\dd}{\dd x},
  \end{equation}
  or more precisely $W[\psi](x)=V[\psi'](x)=\psi'(x)/x$ if $x\neq 0$ and $W[\psi](0)=\psi''(0)$ is linear and continuous with $\|W[\psi]\|_m\leqslant a_m\|\psi\|_{m+2}$ for all $m\in \N$ and all $\psi\in \schonerad$.
\end{lemma}

\begin{proof}
If $\psi\in \schonerad$, then $\psi$ is even and $\psi'$ is odd. Therefore, $\psi'(0)=0$ and $V[\psi']$ is well-defined by Lemma \ref{Lemma: V operator continous}. The function $W[\psi](x)=\psi'(x)/x$ is again even, so that $W$ is well-defined and we can estimate the $m$ norm of $W[\psi]$ by Lemma \ref{Lemma: V operator continous} as
\begin{equation*}
    \|W[\psi]\|_m=\|V(\psi')\|_m\leqslant a_m \|\psi'\|_{m+1}
    \leqslant a_m \|\psi\|_{m+2}.\qedhere
\end{equation*}
\end{proof}

\textbf{Proof of Theorem \ref{prop def of avg} i).}
The directional derivative of $\psi\in\schonerad$ in direction $\xi\in \sphere$ is
\begin{equation*}
    \partial_\xi (\psi\circ \norm_2)(0)
    =\lim_{t\to 0} \frac{\psi(\|t\xi\|)-\psi(\|0\|)}{t}
    =\lim_{t\to 0} \frac{\psi(|t|)-\psi(0)}{t}
    =\psi^\prime(0)=0,
\end{equation*}
and, for $x\neq0$, by the chain rule,
\begin{equation*}
    \partial_\xi(\psi\circ \norm_2)(x)
    =
    \psi'(\|x\|) \inn{\frac{x}{\|x\|},\xi}.
\end{equation*}
Recalling the definition of $V$ in \eqref{eq:V}, we have the representation 
\begin{equation}\label{eq:rep of partial deriv of f circ norm}
    \partial_\xi(\psi\circ\norm_2)(x)=\langle x,\xi\rangle V[\psi](\|x\|)\forrall x\in \R^d.
\end{equation}
In the following, we show inductively that for $\alpha\in \N^d$ there are polynomials $p_{\alpha,1},\ldots, p_{\alpha,|\alpha|}$ of degree $\leqslant k$ such that for every even, smooth function $\psi\colon \R\to \R$ it holds
\begin{equation}\label{eq:Da W}
    D^\alpha(\psi\circ \norm_2)(x)=\sum_{k=1}^{|\alpha|} p_{\alpha,k}(x)\cdot W^k[\psi](\|x\|) \forrall x\in \Rd,
\end{equation}
where $W^k=W\circ W\circ\ldots\circ W$ exactly $k$ times, see \eqref{eq:W}.
For $|\alpha|=1$, we write $\alpha=e_l$ and choose $p_{\alpha,1}(x)=x_l$. By \eqref{eq:rep of partial deriv of f circ norm}, it holds for all $x\in \Rd$ that
\begin{equation*}
    D^\alpha (\psi\circ\norm_2)(x)
    =x_k\frac{\psi'(\|x\|)}{\|x\|}=
    p_{\alpha,1}(x)W[\psi](\|x\|).
\end{equation*}
Now assume \eqref{eq:Da W} holds for $|\alpha|<n$ and choose $\beta\in \N^d$ with $|\beta|=n$. Find some nonzero entry $l$ of $\beta$ and define $\alpha\coloneqq \beta-e_l$.
Since $|\alpha|<n$, we can apply the induction hypothesis and obtain with Lemma \ref{lemma: W operator continuous}
\begin{align*}
    D^\beta(\psi\circ\norm_2)(x)
    &=\partial_{e_l}(D^\alpha(\psi\circ \norm_2))(x)
    =\partial_{e_l} \sum_{k=1}^{|\alpha|} p_{\alpha,k}(x)W^k[\psi](\|x\|)\\&
    =\sum_{k=1}^{|\alpha|} \left( \partial_{e_l} p_{\alpha,k}(x)W^k[\psi](\|x\|)+p_{\alpha,k}(x)x_l\frac{W^k[\psi]'(\|x\|)}{\|x\|} \right)\\&
    = \sum_{k=1}^{|\alpha|} \partial_{e_l} p_{\alpha,k}(x)W^k[\psi](\|x\|)+\sum_{k=2}^{|\alpha|+1} p_{\alpha,k-1}(x)x_lW^k[\psi](\|x\|)\\&
    =\sum_{k=1}^{|\beta|} p_{\beta, k}(x)W^k[\psi](\|x\|),
\end{align*}
where we set $p_{\beta,1}(x)\coloneqq \partial_{e_l} p_{\alpha,1}(x)$, $p_{\beta,|\beta|}(x)\coloneqq p_{\alpha,|\alpha|}(x)x_l$, and 
$
p_{\beta,k}(x)\coloneqq \partial_{e_l} p_{\alpha,k}(x)+p_{\alpha,k-1}(x)x_l
$ for $k=2,\ldots, |\beta|-1.$
This shows \eqref{eq:Da W}.

Let $n\in \N$. Since $\deg p_{\alpha,k}\leqslant k$ for all $\alpha\in \N^d$ we can find $c_n>0$ such that
\begin{equation*}
    |p_{\alpha, k}(x)|\leqslant c_n(1+\|x\|)^n \forrall |\alpha|\leqslant n, \: k=1,\ldots, |\alpha|.
\end{equation*}
For any $x\in \Rd$, we have by the continuity of $W$ in Corollary \ref{lemma: W operator continuous}
\begin{align*}
&\quad \bigl\|\psi\circ\norm_2\bigr\|_n
=\sup_{|\alpha|\leqslant n} \left\| (1+\|\cdot\|)^n D^\alpha(\psi\circ \norm_2)\right\|_\infty\\&
\leqslant \sup_{|\alpha|\leqslant n} \left\| (1+\|\cdot\|)^n \sum_{k=1}^{|\alpha|} p_{\alpha,k}(x)W^k[\psi]\circ\norm\right\|_\infty
\\&
\leqslant c_n \sup_{m=0,\ldots, n} \sum_{k=1}^m \left\|(1+|\cdot|)^{2n} W^k[\psi]\right\|_{\infty}
\\&
\leqslant c_n\sum_{k=1}^n \|W^k[\psi]\|_{2n}
\leqslant c_n \sum_{k=1}^n \prod_{j=1}^k a_j \|\psi\|_{2n+2k}
\leqslant {n} c_n \prod_{j=1}^n a_j \|\psi\|_{4n}.
\end{align*}
Setting $b_n\coloneqq {n} c_n \prod_{j=1}^n a_j$ finishes the proof.
\hfill\qedsymbol{\parfillskip0pt\par}

\begin{lemma}\label{lemma rotated Schwartz function remains Schwartz}
    Let $\varphi\in \schd$ be a Schwartz function and $Q\in \mathrm{O}(d)$, then $\varphi\circ Q\in \schd$ and $\|\varphi\circ Q\|_m\leqslant d^m\|\varphi\|_m $ for all $m\in\N$.
\end{lemma}

\begin{proof}
For $k\in[d]$ and $z\in\Rd$, the chain rule implies 
\begin{align*}
    |\partial_{e_k}(\varphi\circ Q)(z)|&
    =|\langle D(\varphi\circ Q)(z), e_k\rangle |
    =|\langle \nabla \varphi(Q(z))\cdot Q, e_k\rangle |
    =|\langle (\nabla \varphi \circ Q)(z), Q_k\rangle|\\&
    \leqslant 
    \sum_{j=1}^d  |\partial_{e_j} \varphi(Qz)|\cdot |Q_{k,j}|
    \leqslant 
    \sum_{j=1}^d|\partial_{e_j} \varphi(Qz)|
    \leqslant d\cdot \max_{j=1,\ldots,d} |\partial_{e_j} \varphi(Qz)|.
\end{align*}
For $\beta\in \N^d$ we obtain inductively
\begin{equation*}
|D^\beta (\varphi\circ Q)(z)|
\leqslant d^{|\beta|} \max_{|\beta'|=|\beta|} |D^{\beta'}\varphi(Qz)|.
\end{equation*}
Finally,  we have for $m\in \N$ that
\begin{align*}
\|\varphi\circ Q\|_m&
=\max_{|\beta|\leqslant m} \sup_{x\in\R^d} | (1+\|c\|)^m D^\beta [\varphi\circ Q](x)|\\&
\leqslant \max_{|\beta|\leqslant m} d^{|\beta|} \sup_{x\in\R^d}  \|(1+\|x\|)^m (D^\beta \varphi)(Qx)|
=d^m \|\varphi\|_m.
\end{align*}
Therefore $\varphi\circ Q\in \schd$.
\end{proof}

\begin{lemma} \label{lem:phi_xi}
    Let $\varphi\in \schd$ and $\xi \in \sphere$. Then 
    $$\varphi_\xi\colon \R\to \R,\quad r\mapsto \varphi(r \xi)$$ 
    is a one-dimensional Schwartz function and $\|\varphi_\xi\|_m\leqslant d^m \|\varphi\|_m$ for all $m\in\NN$.
\end{lemma}

\begin{proof}
First, let $\xi=e_1$ and $n\in\NN$. we see that $\varphi_{e_1}$ is smooth with
$$D^n \varphi_{e_1}(r) = \left(\tfrac{\dd}{\dd r}\right)^n \varphi(r e_1) = D^\beta \varphi(r e_1),$$
where $\beta = n e_1 \in \NN^d$.
Therefore, we obtain for $m\in\NN$ that
\begin{equation*}
    \|\varphi_{e_1}\|_m
    =\sup_{x_1\in \R}\left|(1+|x_1|)^m D^\beta \varphi(x_1e_1)\right|\
    \leqslant \sup_{x\in \Rd} \left|(1+\|x\|)^m D^\beta \varphi(x)\right|
    \leqslant \|\varphi\|_m.
\end{equation*}
For arbitrary $\xi\in \sphere$, we take some $Q\in  \mathrm{O}(d)$ such that $Qe_1=\xi$, then $\varphi_\xi =(\varphi\circ Q)_{e_1}$. By Lemma \ref{lemma rotated Schwartz function remains Schwartz}, we obtain $\varphi_\xi\in \schone$ and
\begin{equation*}
    \|\varphi_\xi\|_m
    =\|(\varphi\circ Q)_{e_1}\|_m
    \leqslant \|\varphi\circ Q\|_m
    \leqslant d^m \|\varphi\|_m.\qedhere
\end{equation*}
\end{proof}

\textbf{Proof of Theorem \ref{prop def of avg} ii).}
Since $\sphere $ is compact and $\varphi$ is continuous, $\mathcal A_d\varphi(r)$ is well-defined.
For distinct $r,s\in \R$, we can find $t\in \R$ between $r$ and $s$ by the mean value theorem, such that
\begin{equation*}
    \left|\frac{\varphi_\xi(r)-\varphi_\xi(s)}{r-s}\right|
    = | \varphi_\xi'(t)|
    \leqslant \|\varphi_\xi\|_1
    \leqslant d\|\varphi\|_1.
\end{equation*}
Therefore we get by Lebesgue's dominated convergence theorem that
\begin{align*}
    \frac{\dd}{\dd r}\mathcal A_d\varphi(r)&
    =\lim_{s\to r} \frac{1}{\omega_{d-1}}\int_\sphere \frac{\varphi_\xi(s)-\varphi_\xi(r)}{r-s}\d\xi
    = \frac{1}{\omega_{d-1}}\int_\sphere \lim_{s\to r} \frac{\varphi_\xi(s)-\varphi_\xi(r)}{r-s}\,\d\xi\\&
    = \frac{1}{\omega_{d-1}}\int_\sphere \varphi_\xi'(r)\,\d\xi .
\end{align*}
Inductively, we obtain with Lemma~\ref{lem:phi_xi} for any $r\in \R$ and $n\leqslant m$, that
\begin{align*}
    \left|(1+|r|)^m \mathcal A_d\varphi^{(n)}(r)\right|&
    \leqslant \frac{(1+|r|)^m}{\omega_{d-1}}\int_\sphere \left| \varphi_\xi^{(n)}(r)\right|\,\d\xi 
    \\&
    \leqslant \frac{1}{\omega_{d-1}} \int _\sphere \|\varphi_\xi\|_m\,\d\xi
    \leqslant d^m \|\varphi\|_m.
\end{align*}
Clearly $\mathcal A_d\varphi$ is an even function and therefore in $\schdrad$. \hfill \qedsymbol{\parfillskip0pt\par}

\section{Proof of Theorem \ref{thm:dist slicing}} \label{app:proof-new}
Since $F$ is defined via \eqref{eq:basis_function}, we conclude that $F$ is continuous and slowly increasing.
The Fourier transform of $\varphi_{d,m,z}(x)= (m/\pi)^{d/2}\e^{-m\|x-z\|^2}$ defined in Lemma~\ref{lem: approx id} is given by $\hat \varphi_{d,m,z}(v) = \e^{-2\pi\i\langle z,v\rangle} \e^{-\pi^2\|v\|^2/m}$.
Now we obtain
\begin{align*}
    &\quad (\mathcal F_1^{-1}\circ\mathcal A_d\circ \mathcal F_d)[\varphi_{d,m,z}](s)
= \int_\R \e^{2\pi \i sr} \frac{1}{\omega_{d-1}} \int_\sphere \e^{-2\pi  \i \langle z,r\xi \rangle} \e^{-\pi^2\frac{r^2}{m}} \d \xi \d r\\&
= \frac{1}{\omega_{d-1}} \int_\sphere  \int_\R \e^{-2\pi  \i r( \langle z,\xi \rangle-s)} \e^{-\pi^2\frac{r^2}{m}} \d r \d \xi
=\frac{1}{\omega_{d-1}} \int_\sphere  \mathcal F_1\left[\e^{-\pi^2\frac{(\cdot)^2}{m}}\right](\langle z,\xi\rangle -s)\d\xi \\&
    =\E_{\xi\sim\cU_\sphere} [\sqrt{\nicefrac{m}{\pi}}\e^{-m|\langle z,\xi\rangle -s|^2}].
\end{align*}
Then we derive
\begin{equation*}
    \langle \Phi,\varphi_{d,m,z}\rangle =\E_{\xi\sim\cU_\sphere} [\langle f,\varphi_{1,m,\langle z,\xi\rangle}\rangle ].
\end{equation*}
Since $f$ is continuous, even and slowly increasing, we obtain from \eqref{eq:lebesgues_estimate} in the proof of Theorem~\ref{theorem distributional slicing} that there is a slowly increasing function $g\in\mathcal C(\Rd)$ that is independent of $m$ and $\xi$ such that 
\begin{equation}\label{eq:g_majorant}
|\langle f,\varphi_{1,m,\langle z,\xi\rangle}\rangle |\leqslant g(z)\text{ for all } z\in \Rd.
\end{equation}
Therefore, we can apply Lebesgue's dominated convergence theorem and Lemma~\ref{lem: approx id} to conclude
\begin{equation*}
    \lim_{m\to \infty} \langle \Phi,\varphi_{d,m,z}\rangle = \E_{\xi\sim\cU_\sphere} [f(|\langle \xi,z\rangle|)]=F(\|z\|).
\end{equation*}
Let $\varphi\in \mathcal S(\R^d)$ be arbitrary. By \cite[Thm.~4.45]{PlPoStTa23}, the convolution $\varphi_{d,m,0}*\Phi$ is a regular tempered distribution generated by the function $z\mapsto \langle \Phi,\varphi_{d,m,z}\rangle $. Since the convolution with the approximate identity $\varphi_{d,m,0}$ converges in $\mathcal{S}(\mathbb R^d)$ by \cite[Thm 2.3]{AlvAlvLle21}, we obtain
\begin{align*}
\langle \Phi,\varphi\rangle &
= \lim_{m\to \infty} \langle \varphi_{d,m,0} *\Phi,\varphi\rangle 
=\lim_{m\to \infty} \int_\Rd \langle \Phi,\varphi_{d,m,z}\rangle \varphi(z)\d z
= \int_\Rd F(\|z\|)\varphi(z)\d z.
\end{align*}
In the last step, we employed again Lebesgue's dominated convergence theorem with the majorant $\omega_{d-1} g(z)$.
Hence, $\Phi$ is a regular, continuous, even and slowly increasing distribution.
The validity of \eqref{eq:distributional_slicing} follows because $\mathcal R_d^\star$ and the Fourier transform are bijective on radial distributions by Lemma \ref{Lemma Rot inverse to avg}.  \hfill\qedsymbol{\parfillskip0pt\par}

\section{Riemann--Liouville Fractional Integrals and Derivatives} \label{app:sturm_liouv}

In this section, we briefly describe the relation between Abel-type integrals 
\eqref{eq:basis_function_intro} 
and Riemann--Liouville fractional integrals 
together with their inversion formula via fractional derivatives. For more information, we refer e.g. to \cite{Rub24,SKM1993}. 

Denote by $\mathcal{AC}_{\textup{loc}}([0,\infty))$ the space of functions that are absolutely continuous on any interval $[0,b]$ with $b>0$. In particular, $f\in\mathcal{AC}_{\textup{loc}}([0,\infty))$ if and only if there exists $g\in L^1_\mathrm{loc}([0,\infty))$ with $f(x)=\int_0^x g(t) \d t+f(0)$ for all $x\ge0$. Moreover, for $n\ge 1$, the space $\mathcal{AC}^n_\textup{loc}([0,\infty))$ consists of all functions $f\in \mathcal C^{(n-1)}([0,\infty))$, such that $f^{(n-1)}\in \mathcal{AC}_\textup{loc}([0,\infty))$.
Let $\alpha > 0$. Then, 
for $ g \in L^1_{{\textup{loc}}}([0,\infty))$,
the \emph{Riemann--Liouville fractional integral} is given by
  \begin{equation*}
    (I_{+}^\alpha g)(s)
    \coloneqq
    \frac{1}{\Gamma(\alpha)}
    \int_0^s g(t) \, (s-t)^{\alpha-1} \,\d t, \forrall s \in (0,\infty),
  \end{equation*}
  	and the
	\emph{fractional derivative}  of $g\in \mathcal{AC}^n_\textup{loc}([0,\infty))$ by
  \begin{equation*}
    (D_{+}^\alpha G)(s)
    \coloneqq
    \frac{1}{\Gamma(n - \alpha)}
    \, \Bigl( \frac{\diff}{\diff s} \Bigr)^n
    \int_0^s G(t) \, (s - t)^{n - \alpha - 1} \diff t,
  \end{equation*}
where $n = \lfloor \alpha \rfloor + 1$, see \cite[Thm.~2.4]{SKM1993}.
If $\alpha\in\N_0$, then $D^\alpha_+$ is the $\alpha$-th derivative.
For $g \in L^1_{\textup{loc}}([0,\infty))$
and $G \in I_{+}^\alpha(L^1_{\textup{loc}}([0,\infty)))$,
it holds
\begin{equation} \label{eq:Riemann-Liouville-inversion}
    D_{+}^\alpha I_{+}^\alpha g = g        \quad\text{and}\quad
    I_{+}^\alpha D_{+}^\alpha G = G.
\end{equation}

By \cite[Thm 2.3]{SKM1993}, we have $f\in I_+^\alpha(L^1_{\textup{loc}}[0,\infty))$ if and only if
\begin{equation}\label{eq:cond_for_IalphaL1}
    \tilde f \coloneqq
    I_+^{\lfloor\alpha\rfloor+1-\alpha} f\in \mathcal{AC}_{\textup{loc}}^{\lfloor\alpha\rfloor+1}([0,\infty))
    \quad\text{and}\quad
    \tilde f^{(k)}(0)=0 \ \forall k=0,\dots,{\lfloor\alpha\rfloor},
\end{equation}
where the first condition is fulfilled if $f\in\mathcal{AC}_\textup{loc}^{\lfloor\alpha\rfloor+1}([0,\infty))$, see \cite[p 37]{SKM1993}.

\begin{theorem}\label{thm:enusre_slicing}
For $d\ge 3$, let $F\in \mathcal C^{\lfloor \nicefrac{d}{2} \rfloor}([0,\infty))$ and $G(t)\coloneqq F(\sqrt{t})t^{\nicefrac{(d-2)}{2}}$, $t \ge0$.
Then the function
\begin{equation} 
    f(s) 
    \coloneqq
    \frac{2s}{c_d \Gamma(\nicefrac{(d-1)}{2})} \big(D_+^{\nicefrac{(d-1)}{2}} G \big)(s^2), \quad s\ge0,
\end{equation}
is in $C([0,\infty))$ and fulfills
    $F(\|z\|) = \mathbb E_{\xi \sim \mathcal U_\sphere} [f(|\langle \xi,z\rangle|)]$.
Moreover, if the $\lfloor \nicefrac{d}{2}\rfloor$-th derivative of $F$ is slowly increasing, then $f$ is slowly increasing.
\end{theorem}

\begin{proof}
Set $\alpha\coloneqq  \nicefrac{(d-1)}{2}$ 
and $n\coloneqq \lfloor\alpha\rfloor+1$. 

1. First, we show that $I_+^\alpha D_+^\alpha G=G$.
To this end, we have to ensure that \eqref{eq:cond_for_IalphaL1} holds true.
Since $ F\in \mathcal C^{\lfloor \nicefrac{d}{2}\rfloor}([0,\infty))$, we obtain
\begin{equation}\label{eq:recur_deriv}
    G'(s)
    =
    \tfrac{1}{2} F'(\sqrt{s}) s^{\nicefrac{(d-3)}{2}}
   + \tfrac{d-2}{2}  F(\sqrt{s}) s^{\nicefrac{(d-4)}{2}},
\end{equation}
which has an integrable singularity in $s=0$ if $d=3$.
We distinguish two cases.

Let $d$ be odd, so that $\alpha$ is an integer,
$$\lfloor \nicefrac{d}{2}\rfloor = \nicefrac{(d-1)}{2}=\alpha =n-1,$$ 
and $F\in \mathcal C^{n-1}([0,\infty))$. 
By recursively applying \eqref{eq:recur_deriv}, the $(n-1)$-th derivative of $G$ exists and is continuous on $(0,\infty)$ with an integrable singularity $s^{-\nicefrac{1}{2}}$ at $0$. Therefore, $G\in \mathcal{AC}^\alpha_\textup{loc}([0,\infty))$. 
The function $I^{n-\alpha}_+G=I_+^1G$ is the antiderivative of $G$ and therefore we have $I^{n-\alpha}_+G\in \mathcal{AC}^n([0,\infty)) $. 
Concerning the second condition in \eqref{eq:cond_for_IalphaL1}, we have for $k=0,\ldots, n-1$, that $(I^1_+G)^{(k)}(0)=G^{(k-1)}(0)=0$.
Hence, \eqref{eq:cond_for_IalphaL1} is fulfilled and $I_+^\alpha D_+^\alpha G=G$.

Let $d$ be even, so that $n = \nicefrac{d}{2}$ and $F\in \mathcal C^n([0,\infty))$. 
We can differentiate $G$ by recursively applying \eqref{eq:recur_deriv} at least $(n-1)$ times continuously, where the $(n-1)$-th derivative has the form $G^{(n-1)}(s)= \tilde G(\sqrt{s})$ with $\tilde G\in \mathcal C^1([0,\infty))$.
Hence $G^{(n)}(s)=\nicefrac{1}{(2\sqrt{s})}\tilde G'(\sqrt{s})\in L^1_\textup{loc}([0,\infty))$, which yields $G\in \mathcal{AC}^n_\textup{loc}([0,\infty))$. By \cite[p~37]{SKM1993}, the first condition in \eqref{eq:cond_for_IalphaL1}, namely $I_+^{n-\alpha}G=I_+^{\nicefrac{1}{2}}G\in \mathcal{AC}^n_\textup{loc}([0,\infty))$, is satisfied.
For the second condition in \eqref{eq:cond_for_IalphaL1}, we show that $I_+^{\nicefrac{1}{2}}G(s)\in\mathcal O(s^{\nicefrac{(d-1)}{2}})$ for $s\searrow0$, 
which implies that $(I_+^{\nicefrac{1}{2}}G)^{(k)}(0)=0$ for $k=0,\ldots, n-1$.
Using the substitution $t=sx$, we obtain
\begin{align*}
\frac{\sqrt{\pi}}{s^{\frac{d-1}{2}}}I_+^{\nicefrac{1}{2}}G(s)&
=s^{\frac{1-d}{2}} \int_0^s \frac{F(\sqrt{t}) {t}^{\frac{d-2}2}}{\sqrt{s-t}}\d t
\\&
=\int_0^1 F(\sqrt{xs} )\frac{{x}^{\frac{d-2}2}}{\sqrt{1-x}}\d x\xrightarrow{s\to 0} F(0)\int_0^1\frac{{x}^{\frac{d-2}2}}{\sqrt{1-x}}\d x<\infty .
\end{align*}
In summary, we have $I_+^\alpha D_+^\alpha G=G$.

2. Next, we  define 
$$g\coloneqq D_+^\alpha G \quad\text{and}\quad f(s)\coloneqq s \, g(s^2) =  s \, \big(D_+^\alpha G \big)(s^2),$$ 
which are both in $\mathcal C([0,\infty))$. By \eqref{eq:Riemann-Liouville-inversion}, it holds $I_+^\alpha g = G$ and we obtain with the substitution $t=sx^2$ that
\begin{align*}
F(\sqrt{s})&
=s^{-\frac{d-2}{2}} I^\alpha_+ g(s)
=\frac{1}{\Gamma(\alpha)} s^{\frac{2-d}{2}} \int_0^s g(t)(s-t)^\frac{d-3}{2}\d t\\&
=\frac{1}{\Gamma(\alpha)} s^{\frac{2-d}{2}} \int_0^s \frac{f(\sqrt{t})}{\sqrt{t}}(s-t)^\frac{d-3}{2}\d t
=\frac{2}{\Gamma(\alpha)}\int_0^1 f(x\sqrt{s})(1-x^2)^\frac{d-3}{2}\d x.
\end{align*}
By Theorem \ref{thm:slicing_b} this implies
$
    F(\|z\|)=\frac{2}{c_d\Gamma(\alpha)} \mathbb E_{\xi \sim\mathcal U_\sphere} [f(|\langle \xi,z\rangle|)]
$
and we have proved the first assertion.

3.
Finally, let $F^{(\lfloor d/2 \rfloor)}$ be slowly increasing.
It remains to show that $f$ is slowly increasing, too.
If $d$ is odd, then $F^{(k)}$ for $k\le \alpha$ is slowly increasing, and thus also $D_+^\alpha G=G^{(\alpha)}$ is slowly increasing.
If $d$ is even, then  $F^{(k)}$ for $k\le n$ is slowly increasing and thus also $G^{(n)}$.
From the equality
\begin{align*}
\sqrt{\pi}D_+^\alpha G(s)&
=\sqrt{\pi}I_+^{\nicefrac{1}{2}}D_+^nG(s)
=\int_0^s \frac{G^{(n)}(t)}{\sqrt{s-t}} \, \d t
=\sqrt{s}\int_0^1 \frac{G^{(n)}(sx)}{\sqrt{1-x}} \, \d x, 
\end{align*}
we see that also $D_+^\alpha G$ is slowly increasing.
In both cases, $g=D_+^\alpha G$ is slowly increasing and thus also $f(t)=g(t^2)t$ is.
\end{proof}

\end{document}